\documentclass[11pt, a4paper]{article}
\usepackage{mathrsfs}
\usepackage{CJK}
\usepackage{bbm}
\usepackage{stmaryrd}
\usepackage{amsmath}
\usepackage{cases}
\usepackage{amsfonts}
\usepackage{latexsym,amssymb,amsmath,mathrsfs,amsbsy, amsthm}
\usepackage[usenames]{color}
\usepackage{xspace,colortbl}
\usepackage[dvipdfm,
pdfstartview=FitH, CJKbookmarks=true, bookmarksnumbered=true,
bookmarksopen=true,
citecolor=blue 
,colorlinks=true
,pdfborder=1 
,pdfstartpage=2 
,pdfstartview=Fit]{hyperref} 
\allowdisplaybreaks

\newtheorem{thm}{Theorem}[section]
\newtheorem{lem}[thm]{Lemma}

\newtheorem{rem}{Remark}
\newtheorem{coro}[thm]{Corollary}

\newtheorem{defnm}{Definition}[section]

\newcommand{\nn}{\nonumber}

\newcommand{\la}{\langle}
\newcommand{\ra}{\rangle}

\newcommand{\al}{\alpha}
\newcommand{\nb}{\nabla}
\newcommand{\epa}{\varepsilon}

 \newcommand{\Norm}[1]{\left\Vert#1\right\Vert}
\numberwithin{equation}{section}
\begin{document}
\title{\bf\Large Geometric Schr\"odinger-Airy Flows on K\"ahler Manifolds}

\author{\bf Xiaowei Sun and Youde Wang\thanks{Supported
by NSFC, Grant No. 10990013}}

\date{}
\maketitle

\begin{abstract}
We define a class of geometric flows on a complete K\"ahler manifold
to unify some physical and mechanical models such as the motion
equations of vortex filament, complex-valued mKdV equations,
derivative nonlinear Schr\"odinger equations etc. Furthermore, we
consider the existence for these flows from $S^1$ into a complete
K\"ahler manifold and prove some local and global existence results.
\end{abstract}


\section{Introduction}
Let $(N, J, h)$ be a complete K\"ahler manifold with complex structure $J$
and metric $h$. In this paper, we first introduce a class of new
geometric flows from a circle $S^1$ or $\mathbb{R}$ into a complete
K\"ahler manifold $N$. We will see that these flows are of strong
physical and mechanical background and can be seen as the natural
generalization or extension of some physical and mechanical models,
for instance, the motion equations of vortex filament, complex-valued
mKdV equations, Hirota equations, Schr\"odinger-Airy equations, derivative
nonlinear Sch\"odinger equations etc. Therefore, to study these flows are
of important physical and geometric significance. On the other hand, we
may also provide some useful observation to these physical and mechanical
equations from the view point of geometry. By virtue of the geometric
observation we will employ some methods and techniques from geometric
analysis to approach the existence problems for these flows and want
to prove some results on the local and global existence for these geometric flows.

\subsection{The definition of Schr\"odinger-Airy flows and background}

For any smooth map $u(x,t)$ from $S^1\times\mathbb{R}$ into $(N,J,h)$,
Let $\nabla_x$ denote the covariant derivative $\nabla_{\frac{\partial}{\partial x}}$
on the pull-back bundle $u^{-1}TN$ induced from the Levi-Civita connection $\nabla$ on
$N$. For the sake of convenience, we always denote $\nabla_xu$ and $\nabla_tu$ by
$u_x$ and $u_t$ respectively. The energy of a smooth map $v: S^1 \rightarrow N$ is
defined as
$$E_1(v) \equiv \frac{1}{2} \int_{S^1}|v_x|^2dx.$$
And the tension field of $v$ is written by $\tau(v)\equiv\nabla_xv_x$.

For the maps from a unit circle $S^1$ or a real line $\mathbb{R}$
into $N$, we define a class of geometric flows, which we would like
to call \textbf{geometric Schr\"odinger-Airy flow}, as follows:
\begin{eqnarray}\label{eq:1.1}
\frac{\partial u}{\partial t} =\al J_u\nb_xu_x+\beta\left( \nabla_x^2 u_x
+{1\over2}R(u_x,J_uu_x)J_uu_x\right)+\gamma |u_x|^2u_x,
\end{eqnarray}
where $\al$, $\beta$ and $\gamma$ are real constants, $R$ is the Riemannian
curvature tensor on $N$ and $J_u\equiv J(u)$.

If $(N, J, h)$ is a locally Hermitian symmetric space, the geometric flow is
an energy conserved system. Moreover, if $\gamma=0$ it also preserves the
following ``pseudo-helicity" quantity
 $$E_2(u)\equiv\int h( \nabla_xu_x,Ju_x) dx.$$

In the case $\alpha=1$ and $\beta=\gamma=0$, the Schr\"odinger-Airy
flow reduces to the Schr\"odinger flow from $S^1\times\mathbb{R}$ or
$\mathbb{R}\times\mathbb{R}$ into a K\"ahler manifold $(N, J, h)$
formulated by (see \cite{BIKT, Uh, DW, YD, GKT, PW})
 $$\frac{\partial u}{\partial t} = J(u)\nabla_x u_x = J(u)\tau(u),$$
which is an infinite dimensional Hamilton system with respect to the
energy functional. In particular, Rodnianski, Rubinstein and Staffilani
in \cite{RRS} established the global well-posedness of the initial value
problem for the Schr\"odinger flow for maps from the real line into K\"ahler
manifolds and for maps from the circle into Riemann surfaces.

If $\alpha=\gamma=0$ and $\beta=1$, (\ref{eq:1.1}) then reduces to the
KdV geometric flow (see \cite{SW} for more details) on a K\"ahler manifold
$(N, J, h)$ formulated by
$$\frac{\partial u}{\partial t} = \nabla_x^2 u_x
+{1\over2}R(u_x,J_uu_x)J_uu_x,$$
which is an infinite dimensional Hamilton system with respect to the
pseudo-helicity functional.

The Schr\"odinger-Airy geometric flow (\ref{eq:1.1}) is a direct extension
to a K\"ahler manifold of the following curve flow which is used to characterize
the motion of vortex filament. The curve flow is about maps $\textbf{u}$ from
$S^1\times\mathbb{R}$ or $\mathbb{R}\times\mathbb{R}$ into Euclidean
space $\mathbb{R}^3$ which satisfy the following evolution equation
\begin{equation}\label{eq:1.3*}
{\partial\textbf{u}\over \partial t}=\alpha\textbf{u}_s\times \textbf{u}_{ss}+
\beta\big[\textbf{u}_{sss}+{3\over2}\textbf{u}_{ss}\times
(\textbf{u}_s\times \textbf{u}_{ss})\big].
\end{equation}
More precisely, in \cite{FM}, Fukumoto and Miyazaki discussed the motion of a thin
vortex filament with axial velocity and reduced the equation of the
vortex self-induced motion to a nonlinear evolution equation which
can be formulated by using the Frenet frame of curve flow as
\begin{equation}\label{eq:1.3!}
\textbf{u}_t=k\textbf{b}+\gamma({1\over2}k^2\textbf{t}+k_s\textbf{n}+k\tau
\textbf{b}).
\end{equation}
Here $\textbf{u}=\textbf{u}(s,t)$ denotes an evolving filament curve
from $\mathbb{R}\times \mathbb{R}$ into $\mathbb{R}^3$ with
arclength parameter $s$ and time $t$, $\textbf{t},\textbf{n}$ and $\textbf{b}$
denote the unit tangent,
normal and binormal vectors of the filament curve respectively; $k$ and $\tau$ denote the
curvature and the torsion of the filament curve respectively,
$k_s=\frac{\partial k}{\partial s}$ and $\gamma$ is a real constant.
It is not difficult to see that by the Frenet-Serret
formulas, (\ref{eq:1.3!}) could also be reformulated as
\begin{equation*}
\textbf{u}_t=\textbf{u}_s\times \textbf{u}_{ss}+
\gamma\big[\textbf{u}_{sss}+{3\over2}\textbf{u}_{ss}\times
(\textbf{u}_s\times \textbf{u}_{ss})\big].
\end{equation*}
After rescaling  with respect to $t$, the equation could be changed to (\ref{eq:1.3*}).
Thus, by Hasimoto transformation
\begin{eqnarray*}\label{eq:1.4}
\Psi=k \exp \big(i\int^s \tau ds'\big),
\end{eqnarray*}
the equation (\ref{eq:1.3*}) would be reduced to the standard Schr\"odinger-Airy (or Hirota)
equation(see \cite{Hi, SW, TU})
\begin{equation}\label{eq:1.8*}
i\Psi_t+\alpha(\Psi_{ss}+{1\over2}|\Psi|^2\Psi)-i\beta(\Psi_{sss}+{3\over2}|\Psi|^2\Psi_s)=0,
\end{equation}
which is a general model for propogation of pulses in an optical fiber.
In the case $\beta=0$, the equation (\ref{eq:1.8*}) reduces to a
cubic nonlinear Schr\"{o}dinger equation and in the case $\alpha=0$,
(\ref{eq:1.8*}) reduces to the modified KdV equation.

To see the inner relation between (\ref{eq:1.1}) and (\ref{eq:1.3*}) more clearly,
differentiating
 (\ref{eq:1.3*}) with respect to $s$ and and letting
$u(x, t)\equiv \textbf{u}_s(s, t)$
we obtain
\begin{equation}\label{eq:sk}
{\partial u\over \partial t}=\alpha u\times u_{ss}+
\beta\left[u_{sss}+{3\over2}\Big(u_{s}\times
(u\times u_{s})\Big)_s\right].
\end{equation}

One could verify that if given a smooth initial map
$u_o=u(s,0)$ into a unit sphere
$S^2$, then the solution $u$ to
(\ref{eq:sk}) will always be on $S^2$, i.e.,
the length of the tangent vector $|\textbf{u}_s|$ is
preserved (see \cite{NT}). So, the equation (\ref{eq:sk}) describes the evolution of a
geometric flow from $\mathbb{R}$ or $S^1$ into $S^2$. Nishiyama and
Tani have shown the time local and global existence of the initial
value problem of (\ref{eq:sk}) in \cite{NT} and \cite{TN}
respectively.

It is not difficult to see that the Schr\"odinger-Airy geometric flow (\ref{eq:1.1}) is a
natural generalization of (\ref{eq:sk}). Indeed, for a map $u(x)$ from $S^1$ or $\mathbb{R}$
into a two dimensional standard unit sphere $S^2$,  $$u\times:T_uS^2\rightarrow T_uS^2$$ is
just the standard complex structure on $S^2$ and $\tau(u)=\nabla_xu_x$ is the tangential part
of $u_{xx}$ on $S^2$. Meanwhile, a simple calculation shows that there hold
\begin{eqnarray}\label{eq:1.12}
R(u_x,J_uu_x)J_uu_x&=&|u_x|^2u_x,\nn
\end{eqnarray}
\begin{eqnarray}\label{eq:1.13}
\nabla_xu_x&=&u_{xx}+\langle u_x, u_x\rangle u,\nn
\end{eqnarray}
\begin{eqnarray}
\nabla^2_xu_x&=&{d\over dx}(\nabla_xu_x)+\langle
u_x,\nabla_xu_x\rangle u\nn\\
&=&u_{xxx}+3\langle u_x,u_{xx}\rangle u+ |u_x|^2u_x\label{eq:1.14}.\nn
\end{eqnarray}
So,
$$\nabla^2_xu_x+\frac{1}{2}R(u_x,J_uu_x)J_uu_x=u_{xxx}+{3\over2}\big(u_x\times(u\times u_x)\big)_x.$$
Hence, for the case $N=S^2$, the geometric Schr\"odinger-Airy flow
(\ref{eq:1.1}) then reduces to
\begin{eqnarray}
\frac{\partial u}{\partial t} &=&\al J_u\nb_xu_x+\beta\left( \nabla_x^2 u_x
+{1\over2}R(u_x,J_uu_x)J_uu_x\right)\nn\\
&=&\alpha u\times u_{xx}+
\beta\left[u_{xxx}+{3\over2}\Big(u_{x}\times
(u\times u_{x})\Big)_x\right],\nn
\end{eqnarray}
which is just the equation (\ref{eq:sk}).

On the other hand, we recall that the derivative nonlinear Schr\"odinger (DNLS) equation
\begin{equation}\label{eq:Der}
iq_t + q_{xx} = i(|q|^2q)_x, \qquad x, t\in\mathbb{R}
\end{equation}
where $q(x, t)$ is a complex-valued function, arises in the study of
wave propagation in optical fibers \cite{Ag} and in plasma physics
\cite{Mj} (see \cite{CYL, Gu, L} for further references). Scattering
and well-posedness for the Cauchy problem of this equation defined
on $\mathbb{R}$ has been studied by many authors (see \cite{Tak} and
references therein). In particular, J. Colliander, M. Keel, G.
Staffilani, H. Takaoka and T. Tao \cite{CT1, CT2} showed that, under
a smallness assumption on the $L^2$ norm of the initial data, this
equation is globally well-posed in the energy space $H^1$.

While the propagation of nonlinear pulses in optical fibers is described to
first order by the nonlinear Schr\"odinger equation, it is necessary
when considering very short input pulses to include higher-order
nonlinear effects. When the effect of self-steepening ($s \neq 0$)
is included, the fundamental equation is

\begin{equation}\label{eq:Der1.2}
u_\tau = i(\frac{1}{2}u_{\xi\xi}+ |u|^2u) - s
(|u|^2u)_\xi
\end{equation}
where $u$ is the amplitude of the complex field envelope, $\xi$ is a
time variable, and $\tau$ measures the distance along the fiber with
respect to a frame of reference moving with the pulse at the group
velocity. Equation (\ref{eq:Der}) is related to equation
(\ref{eq:Der1.2}) by changing variables (see \cite{L}):

$$u(\tau, \xi)=q(x,t)\exp i(\frac{t}{4s^4}-\frac{x}{2s^2}), \quad\quad
\tau=\frac{t}{2s^2},\quad\quad \xi=-\frac{x}{2s} + \frac{t}{2s^3}.$$
It is an integrable equation and the initial value problem on the line can be analyzed
by means of the inverse scattering transform as demonstrated by Kaup and
Newell \cite{KN}.

For a map $u(x)$ from $S^1$ or $\mathbb{R}$
into a two dimensional standard unit sphere $S^2$, the equation
\begin{eqnarray}\label{eq:sa}
u_t=\al u\times u_{xx} + \gamma |u_x|^2u_x = \al J(u)\tau(u) +
\gamma |u_x|^2u_x
\end{eqnarray}
is equivalent to the above derivative nonlinear Schr\"odinger equation. In
fact, we could see this by adopting the generalized Hasimoto
transformation used in \cite{Uh}.

Precisely, if we assume that $N$ is a compact closed Riemann surface
and assume $u(x,t)\in H^k_{Q}(\mathbb{R}, N)$ (the definition of
$H^k_{Q}(\mathbb{R}, N)$ is given in section 1.2) is a smooth
solution of (\ref{eq:sa}) on $\mathbb{R}\times [0,T)$ such that
$u(x,t)\rightarrow Q \in N$ as $x\rightarrow \infty$ . Let
$\{e,Je\}$ denote the orthonormal frame for $u^{-1}TN$ constructed
in the following manner: Fix a unit vector $e_0\in T_{Q}N$, and for
any $t\in [0,T)$, let $e(x,t)\in T_{u(x,t)}N$ be the parallel
translation of $e_0$ along the curve $u(\cdot,t)$, i.e.,
$\nabla_xe=0$ and $\lim\limits_{x\rightarrow \infty}
e(x,t)\rightarrow e_0$. In local conformal coordinates $z$, with
$z(Q)=0$, after fixing the coordinates of $e_0$ to be
$\zeta_0={1\over \lambda(0,0)}$, the coordinates of the vector $e$
are given by $\zeta={1\over\lambda}e^{i\phi}$  where
 $$\phi=\int_{-\infty}^x Im\left({\lambda_z\over\lambda}z_x\right)dx'$$
The expression of $\zeta$ and $\phi$ is derived by solving the
equation $\nabla_xe=0$, i.e.,
$$\zeta_x+2(\log \lambda)_zz_x\zeta=0.$$
Note that since $\{e, Je\}$ is an orthonormal frame, we have that
$\nabla_te=\eta Je$, where $\eta$ is a real-valued function.
Furthermore, in this frame the coordinates of $u_t$ and $u_x$ are
given by two complex valued functions $p$ and $q$. We set as
follows: first let
\begin{eqnarray*}
u_t&=&p_1e+p_2Je,\\
u_x&=&q_1e+q_2Je,
\end{eqnarray*}
where $p_1$, $p_2$, $q_1$, $q_2$ are real-valued functions of
$(x,t)$, and let $p=p_1+ip_2$ and $q=q_1+iq_2$. Since $\nabla_xe=0$,
it is easy to see that $\nabla_xu_x=q_{1x}e+q_{2x}Je$. Thus, by
(\ref{eq:sa}) we have
$$p_1e+p_2Je=\alpha \left(-q_{2x}e+q_{1x}Je\right)+\gamma |q|^2(q_{1}e+q_{2}Je),$$
which is equivalent to
\begin{eqnarray}\label{eq:sa1}
p=i\alpha q_x+\gamma|q|^2q.
\end{eqnarray}
From $\nabla_xu_t=\nabla_tu_x$, we have
\begin{eqnarray}\label{eq:sa2}
p_x=q_t+i\eta q.
\end{eqnarray}
Combining (\ref{eq:sa1}) and (\ref{eq:sa2}), we obtain
\begin{eqnarray}\label{eq:sa3}
q_t=i\alpha q_{xx}+\gamma \left(|q|^2q\right)_x-i\eta q.
\end{eqnarray}

To get the expression of $\eta$, we note that
\begin{eqnarray*}
R(u_t,u_x)e=\nabla_t\nabla_xe-\nabla_x\nabla_te= -\nabla_x(\eta
Je)=-\eta_xJe.
\end{eqnarray*}
On the other hand
\begin{eqnarray*}
R(u_t,u_x)e=(p_1q_2-p_2q_1)R(e,Je)e=K(u)Im(p\bar{q})Je,
\end{eqnarray*}
where $K(u)=R(e,Je,e,Je)=h\big( e,R(e,Je)Je\big)$ is the Gaussian
curvature of $N$ at $u(x,t)$.

Thus we have
\begin{eqnarray}\label{eq:sa4}
\eta_x=-K(u)Im(p\bar{q}).
\end{eqnarray}
Substituting (\ref{eq:sa1}) into (\ref{eq:sa4}) yields
\begin{eqnarray*}
\eta_x=-K(u)Im(i\alpha
q_x\bar{q})=-{\alpha\over2}K(u)\left(|q|^2\right)_x.
\end{eqnarray*}
Integrating this over $(-\infty,x]$ yields
\begin{eqnarray}\label{eq:sa5}
\eta(x,t)&=&-{\alpha\over2}K(u)|q|^2+{\alpha\over2}\int_{-\infty}^x (K(u))_x(x',t)|q|^2(x',t)dx'\nn\\
&&{}-\eta(-\infty,t).
\end{eqnarray}
Thus, combining (\ref{eq:sa3}) and (\ref{eq:sa5}) yields
\begin{eqnarray}\label{eq:sa6}
q_t&=&i\alpha \left(q_{xx}+{K(u)\over2}|q|^2q\right)+\gamma \left(|q|^2q\right)_x\nn\\
&&{}-{i\alpha\over2}q\int_{-\infty}^x
(K(u))_x(x',t)|q|^2(x',t)dx'-iq\eta(-\infty,t).\nn
\end{eqnarray}
Let $\Psi=qe^{i\int_0^t\eta(-\infty,\tau)dt'}$ and we could easily
see that
\begin{eqnarray}\label{eq:sa7}
\Psi_t&=&i\alpha
\left(\Psi_{xx}+{K(u)\over2}|\Psi|^2\Psi\right)+\gamma
\left(|\Psi|^2\Psi\right)_x\nn\\
&&{}-{i\alpha\over2}\Psi\int_{-\infty}^x
(K(u))_x(x',t)|\Psi|^2(x',t)dx'.
\end{eqnarray}
It is easy to see that if $N$ is a Riemann surface with
constant sectional curvature $K$, (\ref{eq:sa7}) is reduced to
\begin{eqnarray*}
\Psi_t&=&i\alpha \left(\Psi_{xx}+{K\over2}|\Psi|^2\Psi\right)+\gamma
\left(|\Psi|^2\Psi\right)_x.
\end{eqnarray*}
Specially, if $N$ is unit sphere $S^2$ then (\ref{eq:sa7}) is just
(\ref{eq:Der1.2}).

It is well-known that the Schr\"odinger-Airy equation \cite{HK, Kod}
reads as
$$u_t+i\lambda_1 u_{xx} +\lambda_2 u_{xxx}+ i\lambda_3|u|^2u
+\lambda_4|u|^2u_x+\lambda_5 u^2\bar{u}_x=0, \qquad x, t\in
\mathbb{R}$$ where $\lambda_1, \lambda_2\in \mathbb{R}$,
$\lambda_2\neq0$, $\lambda_3, \lambda_4, \lambda_5\in\mathbb{C}$ and
$u=u(x,t)$ is a complex valued function. In the case $\al\neq0$,
$\beta\neq0$ and $\gamma\neq0$, by virtue of the above generalized
Hasimoto transformation we can transform the geometric flow on a
closed Riemann surface with constant sectional curvature into a
Schr\"odinger-Airy equation (see \cite{Uh, SW})
\begin{eqnarray*}
\Psi_t&=&i\alpha
\left(\Psi_{xx}+{K\over2}|\Psi|^2\Psi\right)+\beta\left(\Psi_{xxx}+{3K\over2}|\Psi|^2\Psi_x\right)+\gamma
\left(|\Psi|^2\Psi\right)_x.
\end{eqnarray*}
This is why we call the geometric flow as geometric
Schr\"odinger-Airy flow.

\subsection{Main results and some notations}
In this paper, we confine us to the case $\gamma=0$. First, we
discuss the local existence for the Cauchy problem of geometric
Schr\"odinger-Airy flow from $S^1$ into a K\"ahler manifold $(N,J,h)$  defined by
\begin{equation}\label{eq:1.15}
\left\{
\begin{aligned}
&u_t=\alpha J_u\nb_xu_x+\beta\left(\nabla_x^2u_x+{1\over2}R(u_x,J_uu_x)J_uu_x\right),
\quad x\in S^1;\\
&u(x,0)=u_0(x),
\end{aligned}\right.
\end{equation}
where $\al$, $\beta$ are real positive constants. Furthermore, when $N$ is some
kind of special locally Hermitian symmetric spaces, we could obtain some results
on global existence of (\ref{eq:1.15}).
The method we use here is the same as that we employed to discuss the KdV
geometric flow in \cite{SW}. Remark that we only consider the case that the domain
is $S^1$ in this paper and we could get similar results with those about
the KdV flow.

Before stating our main results, we
introduce several definitions on Sobolev spaces with vector
bundle value. Let $(E, M, \pi)$ be a vector bundle with base
manifold $M$. If $(E, M, \pi)$ is equipped with a metric, then we
may define so-called vector bundle value Sobolev spaces as follows:

\begin{defnm}
$H^m(M, E)$ is the completeness of the set of smooth sections with
compact supports denoted by $\{s| \,\,s\in C_0^\infty(M, E)\}$ with
respect to the norm
$$\|s\|^2_{H^m(M, E)}=\sum_{i=0}^m \int_M |\nabla^i s|^2dM.$$
Here $\nabla$ is the connection on $E$ which is compatible with the
metric on $E$.
\end{defnm}

\begin{defnm} Let $\mathbb{N}$ be the set of positive integers. For $m \in\mathbb{N}\cup\{0\}$,
the Sobolev space of maps from $S^1$ into a Riemannian manifold $(N, h)$ is defined by
$$H^{m+1}(S^1;N) = \{u\in C(S^1; N) |\,\, u_x\in H^m(S^1; TN)\},$$
where $u_x \in H^m(S^1; TN)$ means that $u_x$ satisfies
$$\|u_x\|^2_{H^m(S^1; TN)} =\sum^m_{j=0}\int_{S^1}h(u(x))(\nabla^j_xu_x(x),
\nabla_x^ju_x(x))dx < +\infty.$$
\end{defnm}

Similarly, we define
\begin{defnm}
The Sobolev space of maps from $\mathbb{R}$ into a Riemannian
manifold $(N, h)$ is defined by
$$H^{m+1}(\mathbb{R}; N) = \{u\in C(\mathbb{R}; N)|\,\, u_x\in H^m(\mathbb{R}; TN)\},$$
where $u_x \in H_m(\mathbb{R}; TN)$ means that $u_x$ satisfies
$$\|u_x\|^2_{H^m(\mathbb{R};TN)} =\sum^m_{j=0}\int_{\mathbb{R}}h(u(x))(\nabla^j_xu_x(x),
\nabla_x^ju_x(x))dx < +\infty;$$ and
$$H^{m+1}_Q(\mathbb{R};N) = \{u\in C(\mathbb{R}; N)|d_h(u(x), Q)\in L^2(\mathbb{R}), u_x\in H^m(\mathbb{R}; TN)\},$$
where $d_h(u(x), Q)$ denotes the distance between $u(x)$ and $Q$.
\end{defnm}
We usually use $W^{k,p}(M, N)$ to denote the space of Sobolev maps
from $M$ into $N$, and $W^{k,p}(M, \mathbb{R}^l)$ to denote the
space of Sobolev functions.

\noindent Our main results are as follows:

\begin{thm}\label{thm:1.3}
Let $(N, J, h)$ be a complete K\"{a}hler manifold. Then, the local
solutions $u\in L^\infty ([0,T],H^k(S^1, N))$ $($$k \geqslant 4$$)$ of
the Cauchy problem (\ref{eq:1.15}) with the initial map $u_0\in
H^k(S^1, N)$ is unique. Moreover, the local solution is continuous
with respect to the time variable, i.e., $u\in C([0,T],H^k(S^1,
N))$.

\end{thm}

\begin{thm}\label{thm:1.4}
If $(N, J, h)$ is a noncompact complete K\"ahler manifold, then, for
any integer $k \geqslant 4$ the Cauchy problem of (\ref{eq:1.15}) with
the initial value map $u_0\in H^k(S^1, N)$ admits a unique local
solution $u\in C([0,T],H^k(S^1, N))$, where $T=T(N,||u_0||_{H^4})$. Moreover, if the initial value map $u_0\in H^3(S^1,N)$ and $N$
is a complete K\"ahler manifold with $|\nb^lR|\leqslant B_l(l=0,1,2,3)$
where $B_l$ are positive constants, then the Cauchy problem of (\ref{eq:1.15})
admits a local solution $u\in L^\infty([0,T],H^3(S^1,N))$, where $T=T(N,||u_0||_{H^3})$.

\end{thm}

\begin{thm}\label{thm:1.5}
Assume that $(N,J,h)$ is a complete locally Hermitian symmetric
space satisfying
 $$h(R(Y, X)X, R(X, JX)JX)\equiv 0,$$
where $R(\cdot, \cdot)$ is the Riemannian curvature operator on $N$.
Then for any integer $k\geqslant4$ the Cauchy problem
(\ref{eq:1.15}) with the initial map $u_0\in H^k(S^1,N)$ admits a
unique global solution $u\in C([0,\infty),H^k(S^1,N))$.

\end{thm}

\begin{rem} If $N=M_1\times M_2\times \cdots \times M_n$ is a product
manifold where $(M_i,J_i,h^i)$ ($i=1,2,\cdots,n$) are all manifolds
satisfy the conditions in theorem {\ref{thm:1.5}}, then the results in
theorem {\ref{thm:1.5}} still hold true for $N$.
\end{rem}

\begin{rem}
We have shown in \cite{SW} that the identity on Riemannian curvature in Theorem
\ref{thm:1.5} holds on K\"ahler manifolds with constant holomorphic
sectional curvature, complex Grassmannians, the first class of
bounded symmetric domains. The examples of K\"ahler manifolds with
constant holomorphic sectional curvature are $\mathbb{C}^k$, the
flat complex torus $\mathbb{C}T^l$, the complex projective spaces
$\mathbb{C}P^m$, complex hyperbolic spaces $\mathbb{C}H^n$ and the
compact quotient spaces of complex hyperbolic space modulo by a
torsion free discrete subgroup of automorphism group of
$\mathbb{C}H^n$ etc.
\end{rem}

\begin{rem} It seems that the uniqueness results may be true for
the local solution to the Cauchy problem of Schr\"odinger-Airy flow $u\in L^\infty
([0,T],H^3(S^1, N))$. If so, we can also improve the existence
results and the regularity of solution. In particular, the
uniqueness of solutions to the KdV flows from $S^1$ does not depend
on the the parabolic approximation. Maybe, one could find a
different method to improve regularity.
\end{rem}

As in \cite{SW}, we still adopt the parabolic approximation and employ the geometric
energy method developed in \cite{DW,YD} to show these local
existence problems. To prove the global existence we need to exploit the
following conservation laws $E_1(u)$, $E_3(u)$ and semi-conservation law $E_4(u)$ where
\begin{eqnarray}
E_1(u)&\equiv&{1\over2}\int h(u_x,u_x)dx,\nn\\
E_3(u)&\equiv&\int h(\nabla_x u_x,\nabla_x u_x)dx-{1\over4}\int
h\big( u_x,R(u_x,Ju_x)Ju_x\big) dx,\nn\label{eq:1.18}
\end{eqnarray}
and
\begin{eqnarray}
E_4(u)&\equiv&2\int h(\nabla^2_xu_x,\nabla^2_xu_x)dx-3\int h\big(
\nabla_xu_x,R(\nabla_xu_x,u_x)u_x\big) dx\nn\\
&&{}-5\int h\big(\nabla_x u_x,R(\nabla_xu_x,Ju_x)Ju_x\big)
dx\nn\label{eq:1.19}.
\end{eqnarray}
If $N$ is a locally Hermitian symmetric space, for the smooth
solution $u$ to the Cauchy problem (\ref{eq:1.15}) we will establish
the following in Section 3:
\begin{eqnarray}
{d\over dt}E_1(u) = 0,\quad \quad\quad {d\over dt}E_3(u)=\int h\big(
R(\nabla_xu_x,u_x)u_x,R(u_x,Ju_x)Ju_x\big) dx\label{eq:1.20},
\end{eqnarray}
and
\begin{eqnarray}
{d\over dt}E_4(u)&\leqslant&C(N,E_1(u_0),E_3(u_0))(1+E_4)
\label{eq:1.21}.
\end{eqnarray}

From (\ref{eq:1.20}) we deduce that when $N$ is a locally Hermitian
symmetric space satisfying some geometric condition (see Corollary
(\ref{lm:3.3})), there holds true $E_3(u)=E_3(u_0)$. We utilize
these conservation laws with respect to $E_1(u)$ and $E_3(u)$
to get a uniform a priori bound of $||\nabla_xu_x||_{L^2}$
independent of $T$. By virtue of (\ref{eq:1.21}), we obtain the
global existence results.

This paper is organized as follows: In Section 2 we employ the
geometric energy method to establish the local existence of Schr\"odinger-Airy
geometric flow. We know that the conservation and semi-conservation
laws mentioned  before are crucial for us to establish the global
existence of the Cauchy problem of Schr\"odinger-Airy geometric flow. We will give
a detailed calculation in Section 3.  The global existence of Schr\"odinger-Airy
geometric flows on some special K\"ahler manifolds is proved in Section 4.

\section{Local Existence and Uniqueness}

In this section we establish the local existence and the uniqueness of solutions
for the Cauchy problem of the Schr\"odinger-Airy flow (\ref{eq:1.15}) on a K\"{a}hler
manifold $(N,J,h)$
\begin{equation}\label{eq:4.1}
\left\{
\begin{aligned}
&u_t=\al J_u\nb_xu_x+\beta\left(
\nabla_x^2u_x+{1\over2}R(u_x,J_uu_x)J_uu_x\right),\quad x\in S^1;\nn\\
&u(x,0)=u_0(x).
\end{aligned}\right.\nn
\end{equation}

We still use the approximate method as in \cite{SW}
to show the local existence of (\ref{eq:1.15}).
To this end, we discuss the following Cauchy problem:
\begin{equation}\label{eq:4.2}
\left\{
\begin{aligned}
&u_t=-\epa\nabla_x^3u_x+\al J_u\nb_xu_x+\beta\left(
\nabla_x^2u_x+{1\over2}R(u_x,J_uu_x)J_uu_x\right),\quad x\in S^1;\\
&u(x,0)=u_0(x).
\end{aligned}\right.
\end{equation}
where $\epa>0$ is a small positive constant.

 We could imbed $N$
into a Euclidean space $\mathbb{R}^n$ for some large positive
integer $n$. Then $N$ could be regarded as a sub-manifold of
$\mathbb{R}^n$ and $u:S^1\times
\mathbb{R}\rightarrow N\subset \mathbb{R}^n$ could be represented as
$u=(u^1,\cdots,u^n)$ with $u^i$ being globally defined functions on
$S^1$  so that the Sobolev-norms of $u$ make sense. We have
\begin{eqnarray*}
||u||^2_{W^{m,2}}=\sum_{i=0}^m||D^iu||^2_{L^2},
\end{eqnarray*}
where $D$ denotes the covariant derivative for functions on $S^1$.
The equation (\ref{eq:4.2}) then becomes a fourth order parabolic
system in $\mathbb{R}^n$. In the appendix of \cite{SW}, we have
shown that the parabolic equation admits a local solution $u_\epa
\in C([0, \infty), W^{k,2}(S^1, N))$ if the initial value map
$u_0\in W^{k,2}(S^1, N)$ where $k\geqslant3$.

Thus, to prove the local existence of (\ref{eq:1.15}), we just need
to find a uniform positive lower bound $T$ of $T_\epa$
and uniform bounds for various norms of $u_\epa(t)$ in suitable
spaces for $t$ in the time interval $[0,T)$. Once we get these
bounds it is clear that $u_\epa$ subconverge to a strong solution of
(\ref{eq:1.15}) as $\epa\rightarrow 0$.

Throughout this paper, we simply denote $h(X,Y)$ by $\la X,Y\ra$ for
all $X,Y\in \Gamma(u^{-1}TN)$. Note that if $X\in \Gamma(u^{-1}TN)$
we have in local coordinates
$$(\nabla_x X)^\alpha= {\partial X^\alpha\over \partial
x}+\Gamma^\alpha_{\beta\gamma}(u){\partial u^\beta\over \partial
x}X^\gamma $$ and for $X=u_x$ we have
$$(\nabla_t u_x)^\alpha= {\partial^2 u^\alpha\over {\partial t\partial x}}
+\Gamma^\alpha_{\beta\gamma}(u){\partial u^\beta\over \partial
t}{\partial u^\gamma \over \partial x}.$$ It is easy to see that
$\nabla_t u_x=\nabla_x u_t$.

Now let $u=u_\epa$ be a solution of (\ref{eq:4.2}). We have the
following results.
\begin{lem}\label{lm:4.1}
(i) Assume that $N$ is a complete K\"ahler manifold with uniform
bounds on the curvature tensor $R$ and its covariant derivatives of
any order $($i.e., $|\nabla^lR|\leqslant B_l$, $l=0, 1, 2, \cdots$$)$,
and $u_0\in H^k(S^1,N)$ with an integer $k\geqslant3$. Then there
exists a constant $T=T(||u_0||_{H^3})$, independent of
$\epa\in(0,1)$, such that if $u\in C([0,T_\epa),H^k(S^1,N))$ is a
solution of (\ref{eq:4.2}) with $\epa\in(0,1)$, then
$T(||u_0||_{H^3})\leqslant T_\epa$ and $||u(t)||_{H^{m+1}}\leqslant
C(||u_0||_{H^{m+1}})$ for any integer $2\leqslant m\leqslant k-1$.

(ii) Assume that $N$ is a complete K\"ahler manifold and $u_0\in
H^k(S^1,N)$ with an integer $k\geqslant5$. Then there exists a
constant $T=T(||u_0||_{H^5})>0$, independent of $\epa\in(0,1)$, such
that if $u\in C([0,T_\epa),H^k(S^1,N))$ is a solution of
(\ref{eq:4.2}) with $\epa\in(0,1)$, then $T(||u_0||_{H^5})\leqslant
T_\epa$ and $||u(t)||_{H^{m+1}}\leqslant C(||u_0||_{H^{m+1}})$ for
any integer $2\leqslant m\leqslant k-1$.
\end{lem}

\begin{proof}
First fix a $k\geqslant3$ and let
$m$ be any integer with $2\leqslant m\leqslant k-1$.
We may assume that $u_0$ is $C^\infty$ smooth.
Otherwise, we always choose a sequence of smooth functions
$\{u^i_0\}$ such that $u^i_0\rightarrow u_0$ with respect to the
norms $\|\,\cdot\,\|_{H^k}$ where $k\geqslant 3$.

As $N$ may not be compact we let $\Omega\triangleq \{p\in
N:\text{dist}_N(p,u_0(S^1))<1\}$, which is an open subset of
$N$
 with compact closure $\bar{\Omega}$. Let
$$T'=\sup\{t>0:u(S^1,t)\subset\Omega\}.$$
Now prove that for $k=3$,
\begin{eqnarray}\label{eq:4.4}
{d\over dt}||u_x||^2_{H^2}\leqslant C(\Omega,\al,\beta)\sum_{l=2}^4 ||u_x||^{2l}_{H^2},
\end{eqnarray}
for all $t\in[0,T_\epa]$.


We differentiate each term in $||u_x||_{H^2}^2$ with respect to $t$. We have
\begin{eqnarray*}\label{eq:99}
&&{}{d\over dt}\int |u_x|^2dx=2\int \la u_x,
\nabla_t u_x\ra dx \\
&=&2\int\la u_x,\nabla_x u_t\ra dx=-2\int\la \nabla_x u_x,u_t\ra
dx.\nn
\end{eqnarray*}
Substituting (\ref{eq:4.2}) yields
\begin{eqnarray}\label{eq:4.5}
&&{}{d\over dt}\int |u_x|^2dx\nn\\
&=&2\epa\int\la\nabla_x u_x,\nabla_x^3u_x\ra dx-2\alpha\int\la \nabla_x u_x,J\nb_xu_x\ra dx\nn\\
&&{}-\beta\left(2\int\la\nabla_x u_x,\nabla_x^2u_x\ra dx+\int\la\nabla_x u_x,R(u_x,Ju_x)Ju_x\ra dx\right)\nn\\
&=&-2\epa\int|\nabla_x^2 u_x|^2 dx-\beta\int\la\nabla_x u_x,R(u_x,Ju_x)Ju_x\ra dx\nn\\
&\leqslant&-2\epa\int|\nabla_x^2 u_x|^2 dx+C(\Omega,\beta)\int|\nabla_xu_x||u_x|^3 dx\nn\\
&\leqslant& C(\Omega,\beta)||u_x||^4_{H^1}.
\end{eqnarray}

We should remark that each time we substitute (\ref{eq:4.2}) into
the equality there will appear three parts: the first part that
contains $\epa$, the Schr\"odinger part which contains $\alpha$ and
the KdV part which contains $\beta$. We need to estimate each terms
in these three parts. When dealing with terms from the KdV part, we
will use the results in \cite{SW} directly since the calculations
are the same.

Now we consider $\int|\nabla_x u_x|^2dx$. Differentiating it with respect to $t$ yields
\begin{eqnarray}\label{eq:4.51}
{d\over dt}\int| \nabla_x u_x|^2 dx&=& 2\int\la \nabla_x
u_x,\nabla_t\nabla_x u_x\ra dx\nn\\
&=&2\int \la \nabla_x u_x,\nabla_x^2u_t\ra dx+2\int \la \nabla_x u_x,R(u_t,u_x)u_x\ra dx\nn\\
&=&2\int \la \nabla_x^3 u_x,u_t\ra dx+2\int\la u_t, R(\nabla_x
u_x,u_x)u_x\ra dx.\nn
\end{eqnarray}
Thus substituting Eq. (\ref{eq:4.2}) yields
\begin{eqnarray}\label{eq:4.7}
&&{}{d\over dt}\int |\nabla_x u_x|^2 dx\nn\\
&=&-2\epa\int|\nabla_x^3u_x|^2dx-2\epa\int\la\nabla_x^3u_x,R(\nabla_x u_x,u_x)u_x\ra dx\nn\\
&&{}+\al\left(2\int \la \nabla_x^3 u_x,J\nb_xu_x\ra dx+2\int\la J\nb_xu_x, R(\nabla_x
u_x,u_x)u_x\ra dx\right)\nn\\
&&{}+\beta\left(2\int\la\nabla_x^2u_x,\nabla_x^3u_x\ra dx
+\int\la\nabla_x^3u_x,R(u_x,Ju_x)Ju_x\ra dx\right.\nn\\
&&{}+2\int\la\nabla_x^2u_x,R(\nabla_x u_x,u_x)u_x\ra dx\nn\\
&&{}+\left.\int\la R(u_x,Ju_x)Ju_x,R(\nabla_x u_x,u_x)u_x\ra dx\right).
\end{eqnarray}

For the second term of the right hand side, integrating by parts yields
\begin{eqnarray}\label{eq:4.9}
&&{}-2\epa\int\la\nabla_x^3u_x,R(\nabla_x u_x,u_x)u_x\ra dx\nn\\
&=&2\epa\int\la\nabla_x^2u_x,(\nabla_x R)(\nabla_x u_x,u_x)u_x\ra dx+2\epa\int\la\nabla_x^2u_x,R(\nabla_x^2 u_x,u_x)u_x\ra dx\nn\\
&&{}+2\epa\int\la\nabla_x^2u_x,R(\nabla_x u_x,u_x)\nabla_x u_x\ra dx\nn\\
&\leqslant& C(\Omega)(\int|\nabla_x^2u_x||\nabla_xu_x||u_x|^3dx+\int|\nabla_x^2u_x||u_x|^2dx
+\int|\nabla_x^2u_x||\nabla_xu_x|^2|u_x|dx).
\end{eqnarray}

For the Schr\"odinger part on the right of (\ref{eq:4.7}), we have
\begin{eqnarray}\label{eq:4.01}
&&{}\al\left(2\int \la \nabla_x^3 u_x,J\nb_xu_x\ra dx+2\int\la J\nb_xu_x, R(\nabla_x
u_x,u_x)u_x\ra dx\right)\nn\\
&=&-2\al\int \la \nabla_x^2 u_x,J\nb^2_xu_x\ra dx-2\al\int\la \nb_xu_x, R(\nabla_x
u_x,u_x)Ju_x\ra dx\nn\\
&=&-2\al\int\la \nb_xu_x, R(\nabla_x
u_x,u_x)Ju_x\ra dx\nn\\
&\leqslant& C(\Omega,\al)\int|\nb_xu_x|^2|u_x|^2.
\end{eqnarray}

For the KdV part in (\ref{eq:4.7}), after integrating by parts we have
\begin{eqnarray}\label{eq:4.8}
&&{}\beta\left(2\int\la\nabla_x^2u_x,\nabla_x^3u_x\ra dx
+\int\la\nabla_x^3u_x,R(u_x,Ju_x)Ju_x\ra dx\right.\nn\\
&&{}+2\int\la\nabla_x^2u_x,R(\nabla_x u_x,u_x)u_x\ra dx\nn\\
&&{}+\left.\int\la R(u_x,Ju_x)Ju_x,R(\nabla_x u_x,u_x)u_x\ra dx\right)\nn\\
&\leqslant&C(\Omega)(\int|\nabla_x^2u_x||\nabla_xu_x||u_x|^2dx
+\int|\nabla_x^2u_x||u_x|^4dx+\int|\nabla_xu_x||u_x|^5dx).
\end{eqnarray}

Using H\"{o}lder inequality and interpolation inequalities,
 Eqs. (\ref{eq:4.7})$-$(\ref{eq:4.8}) yield
\begin{eqnarray}\label{eq:4.10}
{d\over dt}\int |\nabla_x u_x|^2
dx+2\epa\int|\nabla_x^3u_x|^2dx\leqslant C(\Omega)||u_x||^4_{H^2}.
\end{eqnarray}

Next we compute $\int|\nabla_x^2u_x|^2dx$. We have
\begin{eqnarray}\label{eq:4.11}
&&{}{d\over dt}\int|\nabla_x^2u_x|^2dx=2\int\la\nabla_t\nabla_x^2u_x,\nabla_x^2u_x\ra dx\nn\\
&=&2\int\la\nabla_x\nabla_t\nabla_x u_x,\nabla_x^2u_x\ra dx+2\int\la R(u_t,u_x)\nabla_x u_x,\nabla_x^2u_x\ra dx\nn\\
&=&-2\int\la\nabla_x^2u_t,\nabla_x^3u_x\ra dx-2\int\la\nabla_x^3u_x,R(u_t,u_x)u_x\ra dx\nn\\
&&{}+2\int\la\nabla_x^2u_x,R(u_t,u_x)\nabla_x u_x\ra dx\nn\\
&=&-2\int\la u_t,\nabla_x^5u_x\ra dx-2\int\la u_t,R(\nabla_x^3u_x,u_x)u_x\ra dx\nn\\
&&{}+2\int\la u_t,R(\nabla_x^2u_x,\nabla_x u_x)u_x\ra dx.
\end{eqnarray}
 Substituting (\ref{eq:4.2}) into (\ref{eq:4.11}) while noting that
 $$\int\la J\nb_xu_x,\nb_x^5u_x\ra dx=\int\la J\nb^3_xu_x,\nb_x^3u_x\ra dx=0;$$
 and
 $$\int\la\nabla_x^3u_x,\nabla_x^5u_x\ra dx=-\int|\nabla_x^4u_x|^2dx, \quad\int\la\nabla_x^2u_x,\nabla_x^5u_x\ra dx=0,$$
we have
\begin{eqnarray}\label{eq:4.12}
&&{d\over dt}\int|\nabla_x^2u_x|^2dx+2\epa\int|\nabla_x^4u_x|^2dx\nn\\
&=&2\epa\int \la\nabla_x^3u_x,R(\nabla_x^3u_x,u_x)u_x\ra dx
-2\epa\int\la\nabla_x^3u_x,R(\nabla_x^2u_x,\nabla_x u_x)u_x\ra dx+I_0,
\end{eqnarray}
where
\begin{eqnarray}
I_0&=&-\al\left(2\int \la J\nabla_xu_x,R(\nabla_x^3u_x,u_x)u_x\ra dx
-2\int\la J\nabla_xu_x,R(\nabla_x^2u_x,\nabla_x u_x)u_x\ra dx\right)\nn\\
&&{}-\beta\left(\int\la \nabla_x^5u_x,R(u_x,Ju_x)Ju_x\ra dx
+2\int\la\nabla_x^2u_x,R(\nabla_x^3u_x,u_x)u_x\ra dx\right.\nn\\
&&{}+\int \la R(u_x,Ju_x)Ju_x,R(\nabla_x^3u_x,u_x)u_x\ra dx
-2\int\la\nabla_x^2u_x,R(\nabla_x^2u_x,\nabla_x u_x)u_x\ra dx\nn\\
&&{}-\left.\int\la R(u_x,Ju_x)Ju_x,R(\nabla_x^2u_x,\nabla_x u_x)u_x\ra dx\right).\nn
\end{eqnarray}
For the first term of (\ref{eq:4.12}), integrating by parts yields
\begin{eqnarray}\label{eq:4.13}
&&{}\quad 2\epa\int \la\nabla_x^3u_x,R(\nabla_x^3u_x,u_x)u_x\ra dx\nn\\
&=&-2\epa\int \la\nabla_x^3u_x,(\nabla_x R)(\nabla_x^2u_x,u_x)u_x\ra dx-2\epa\int \la\nabla_x^4u_x,R(\nabla_x^2u_x,u_x)u_x\ra dx\nn\\
&&{}-2\epa\int \la\nabla_x^3u_x,R(\nabla_x^2u_x,u_x)\nabla_x u_x\ra dx
-2\epa\int \la\nabla_x^3u_x,R(\nabla_x^2u_x,\nabla_x u_x)u_x\ra dx.
\end{eqnarray}
Thus, for any $\delta>0$, (\ref{eq:4.13}) together with the second term of
(\ref{eq:4.12}) yields
\begin{eqnarray}\label{eq:4.14}
&&2\epa\int \la\nabla_x^3u_x,R(\nabla_x^3u_x,u_x)u_x\ra dx-2\epa\int\la\nabla_x^3u_x,R(\nabla_x^2u_x,\nabla_x u_x)u_x\ra dx\nn\\
&\leqslant&\epa\delta \int|\nabla_x^4u_x|^2dx+4\epa\delta \int|\nabla_x^3u_x|^2dx\nn\\
&&{}+{\epa\over 2\delta}\left(\int|(\nabla_x R)(\nabla_x^2u_x,u_x)u_x|^2dx
+\int|R(\nabla_x^2u_x,u_x)u_x|^2dx\right.\nn\\
&&{}+\int|R(\nabla_x^2u_x,u_x)\nabla_x u_x|^2dx
\left.+2\int|R(\nabla_x^2u_x,\nabla_x u_x)u_x|^2dx\right)\nn\\
&\leqslant&\epa\delta \int|\nabla_x^4u_x|^2dx+4\epa\delta \int|\nabla_x^3u_x|^2dx
+{C(\Omega)\over2\delta}\int|\nabla_x^2u_x|^2(|u_x|^6+|u_x|^4+|\nabla_x u_x|^2|u_x|^2)dx\nn\\
&\leqslant&\epa\delta \int|\nabla_x^4u_x|^2dx+4\epa\delta
\int|\nabla_x^3u_x|^2dx+{C(\Omega)\over2\delta}(||u_x||^4_{H^2}+||u_x||^6_{H^2}+||u_x||^8_{H^2}),
\end{eqnarray}
where we used the following interpolation inequalities
\begin{eqnarray*}
||u_x||_{L^\infty}&\leqslant& C(\Omega)(||\nabla_x u_x||^2_{L^2}+||u_x||^2_{L^2})^{1\over4}||u_x||^{1\over2}_{L^2};\nn\\
||\nabla_x u_x||_{L^\infty}&\leqslant& C(\Omega)(||\nabla_x^2
u_x||^2_{L^2}+||\nabla_xu_x||^2_{L^2})^{1\over4}||\nabla_xu_x||^{1\over2}_{L^2}.\nn
\end{eqnarray*}
For the left part $I_0$ of (\ref{eq:4.12}),
after integrating by parts repeatedly (see \cite{SW} for detail), we could obtain that
\begin{eqnarray}\label{eq:4.155}
I_0&\leqslant& C(\Omega,\al)\left(\int|\nabla^2_xu_x||\nabla_xu_x|^2|u_x|dx
+\int|\nabla^2_xu_x|^2|u_x|^2dx\right.\nn\\
&&{}+\left.\int|\nabla^2_xu_x||\nabla_xu_x||u_x|^3dx\right)\nn\\
&&{}+C(\Omega,\beta)\left(\int|\nabla^2_xu_x|^2|\nabla_xu_x||u_x|dx
+\int|\nabla^2_xu_x||\nabla_xu_x|^3dx\right.\nn\\
&&{}+\int|\nabla^2_xu_x|^2|u_x|^3dx+\int|\nabla^2_xu_x||\nabla_xu_x|^2|u_x|^2dx\nn\\
&&{}+\left.\int|\nabla^2_xu_x||\nabla_xu_x||u_x|^4dx\right).
\end{eqnarray}
Thus, by H\"{o}lder inequality and interpolation inequalities, we obtain the desired bound
 $$I_0\leqslant C(\Omega,\al,\beta)\sum_{l=2}^4 ||u_x||^{2l}_{H^2}.$$

Hence we have
\begin{eqnarray}\label{eq:4.16}
&&{d\over dt}\int|\nabla_x^2u_x|^2dx+2\epa\int|\nabla_x^4u_x|^2dx\nn\\
&\leqslant&\epa\delta \int|\nabla_x^4u_x|^2dx+4\epa\delta\int|\nabla_x^3u_x|^2dx
+C(\Omega,\al,\beta)({1\over2\delta}+1)\sum_{l=2}^4 ||u_x||^{2l}_{H^2}.
\end{eqnarray}
In view of (\ref{eq:4.5}), (\ref{eq:4.10}), and (\ref{eq:4.16}), we have
\begin{eqnarray}
&&{d\over dt}||u_x||^2_{H^2}+(2-\delta)\epa\int|\nabla_x^4u_x|^2dx+(1-4\delta)\epa\int|\nabla_x^3u_x|^2dx\nn\\
&\leqslant& C(\Omega,\al,\beta)({1\over2\delta}+1)\sum_{l=2}^4 ||u_x||^{2l}_{H^2}.\nn
\end{eqnarray}
Let $\delta={1\over8}$ and we get the desired inequality (\ref{eq:4.4}).\\

For $3\leqslant m\leqslant k-1$, by the similar process, we could
get the following inequality
\begin{eqnarray}\label{eq:4.17}
{d\over dt}||u_x||^2_{H^{m}}\leqslant C(\Omega,
||u_x||_{H^{m-1}},\al,\beta)||u_x||^2_{H^{m}},
\end{eqnarray}
where $C(\Omega, ||u_x||_{H^{m-1}},\al,\beta)$ depends on $\al$, $\beta$, $||u_x||_{H^{m-1}}$
and the bounds on the curvature $R$ and its covariant derivatives
$\nabla^lR$ with $l\leq m+1$ on $\Omega\subset N$. We omit the details of the proof.\\

Let $f(t)=||u_x||^2_{H^2}+1$, then we have
\begin{eqnarray}\label{eq:4.18}
{df\over dt}\leqslant C(\Omega,\al,\beta)f^4,\quad f(0)=||u_{0x}||^2+1.
\end{eqnarray}
It follows from (\ref{eq:4.18}) that there exists constants $T_0>0$
and $C_0>0$ such that
\begin{eqnarray*}
||u_{x}||_{H^2}\leqslant C_0,\quad t\in[0,\text{min}(T_0,T')].
\end{eqnarray*}
Now let $T=\text{min}(T_0,T')$. If $m=3$, by the Gronwall
inequality, we can obtain from (\ref{eq:4.17}):
\begin{eqnarray*}
||u_{ x}||_{H^3}\leqslant C_1(\Omega, T, ||u_{0x}||_{H^3},\al,\beta),\quad
\text{for all}\quad t\in[0,T].
\end{eqnarray*}
Then by induction we have that there exists a constant
$C_{m-2}(\Omega, ||u_{0x}||_{H^m})>0$, such that for any $3\leqslant
m\leqslant k-1$
\begin{eqnarray}\label{eq:4.19}
\text{ess sup}_{t\in[0,T]}||u_x||_{H^m}\leqslant C_{m-2}(\Omega,
||u_{0x}||_{H^m},\al,\beta).
\end{eqnarray}
Since $\Omega$ is compact, consequently $||u(t)||_{L^\infty(S^1)}$
is uniformly bounded for $t\in[0,T]$.\\

If $N$ is of uniform bounds on the curvature tensor and its
derivatives of any order, it is easy to see from the above arguments
that $T=T_0$ since the coefficients of the above differential
inequalities depend only on the bounds on Riemann curvature tensor
$R$ and its covariant derivatives $\nabla^l R$ of some order on $N$.
That is $T=T(S,||u_{0}||_{H^3})$ depends only on $N$, $u_0$, not on
$0<\epa<1$.

Now we consider the case $N$ is a noncompact, complete K\"ahler
manifold without the bounded geometry assumptions. Note that a
positive lower bound of $T'$ can also be derived from
(\ref{eq:4.19}) when $k\geq 5$. Indeed, it is easy to see from the approximate equation
of Schr\"odinger-Airy flow and the interpolation inequality  that (\ref{eq:4.19}) implies
$$\text{ess sup}_{t\in[0,T]}||u_t||_{L^2(S^1, TN)}\leqslant C(\Omega, \|u_{0x}\|_{H^3}).$$
On the other hand, from the approximate equation of Schr\"odinger-Airy flow we have
$$\nabla_xu_t=-\epa\nabla_x^4u_x +\al J\nb_x^2u_x+\beta\left(\nabla_x^3u_x + \frac{1}{2} \nabla_x(R(u_x, Ju_x)Ju_x)\right).$$
Hence, when $k\geq5$ we infer from (\ref{eq:4.19}) and the interpolation
inequality that
$$\text{ess sup}_{t\in[0,T]}||u_t||_{H^1(S^1, TN)}\leqslant C(\Omega, \|u_{0x}\|_{H^4},\al,\beta).$$

However, by the interpolation inequality, for some $0<a<1$ there holds
\begin{eqnarray*}
||u_t(s)||_{L^\infty}\leqslant C||u_t(s)||^a_{H^1}||u_t(s)||^{1-a}_{L^2}.
\end{eqnarray*}
This implies that, for some $\mathcal{M}>0$, there holds true
$$\text{ess sup}_{t\in[0,T]}||u_t||_{L^\infty}\leqslant \mathcal{M}.$$
Thus we have
\begin{eqnarray*}
\sup_{x\in S^1} d_N(u(x,t), u_0(x)) \leqslant \mathcal{M}t, \quad
\text {for}\quad t<T.
\end{eqnarray*}
If $T'>T_0$ we get the lower bound, so we may assume that
$T'\leqslant T_0$. Then letting $t\rightarrow T'$ in the above
inequality we get $\mathcal{M} T'\geqslant1$. Therefore, if we set
$T=\text{min}\{{1\over\mathcal{M}}, T_0\},$ then the desired
estimates hold for $t\in[0,T].$

It is easy to find that the solution to (\ref{eq:4.2}) with
$\epa\in(0,1)$ must exists on the time interval $[0,T]$. Otherwise,
we always extend the time interval of existence to cover $[0,T]$,
i.e., we always have $T_\epa\geqslant T$. Thus we complete the proof
of this lemma.
\end{proof}

\begin{lem}\label{thm:1.1}
If $(N, J, h)$ is a complete K\"ahler manifold with uniform bounds
on the curvature tensor and its covariant derivatives of any order
$($i.e., $|\nabla^lR|\leqslant B_l$, $l=0, 1, 2, \cdots$$)$, then, for
any integer $k\geqslant3$ the Cauchy problem of (\ref{eq:1.15}) with
the initial value map $u_0\in H^k(S^1, N)$ admits a local solution
$u\in L^\infty([0,T],H^k(S^1, N))$, where $T=T(N,||u_0||_{H^3})$.

\end{lem}

Before proving Lemma \ref{thm:1.1}, we remark that in \cite{YD}, Ding and Wang have shown that the $H^{m}$
norm of section $\nabla u$ defined in section 1.3  is equivalent to
the usual Sobolev $W^{m+1, 2}$ norm of the map $u$. Precisely, we
have
\begin{lem}\label{lm:4.2} (\cite{YD})
Assume that $N$ is a compact Riemannian manifold with or without
boundary and $m\geqslant 1$. Then there exists a constant $C =
C(N,m)$ such that for all $u \in C^\infty(S^1,N)$,
\[ \Norm{Du}_{W^{m-1,2}} \leqslant C\sum_{i=1}^m \Norm{\nabla u}^i_{H^{m-1,2}} \]
and
\[ \Norm{\nabla u}_{H^{m-1,2}} \leqslant C\sum_{i=1}^m \Norm{Du}^i_{W^{m-1,2}}. \]\\
\end{lem}

{\bf\textit{Proof of Lemma \ref{thm:1.1}}}. We assume that
$N$ is compact and imbed $N$ into $\mathbb{R}^n$.
If $u_0:S^1\rightarrow N$ is $C^\infty$, then from Lemma
\ref{lm:4.1} we have that the Cauchy problem (\ref{eq:4.2}) admits a
unique smooth solution $u_\epa$ which satisfies the estimates in
Lemma \ref{lm:4.1}. Hence by Lemma \ref{lm:4.1} and Lemma \ref{lm:4.2} we have that for
any integer $p>0$ and $\epa\in(0,1]$:
\begin{eqnarray}\label{eq:ww}
\sup_{t\in[0,T]} ||u_\epa||_{W^{p,2}(N)}\leqslant C_p(N,u_0),
\end{eqnarray}
where $C_p(N,u_0)$ does not depend on $\epa$. Hence, by sending
$\epa\rightarrow 0$ and applying the embedding theorem of Sobolev
spaces to $u$, we have $u_\epa\rightarrow u\in C^p(S^1\times[0,T])$
for any $p$. It is easy to check that $u$ is a solution to the
Cauchy problem (\ref{eq:4.1}).

If $u_0:S^1\rightarrow N$ is not $C^\infty$, but $u_0\in
W^{k,2}(S^1,N)$, we may always select a sequence of
$C^\infty$ maps from $S^1$ into $N$, denoted by $u_{i0}$, such
that
$$u_{i0}\rightarrow u_0\quad\text{in}\quad W^{k,2},\quad\text{as}\quad i\rightarrow \infty.$$
Thus following from Lemma \ref{lm:4.2} we have
\begin{eqnarray*}
||\nabla_x u_{i0}||_{H^{k-1}}\rightarrow ||\nabla_x u_0||_{H^{k-1}},\quad\text{as}\quad i\rightarrow\infty.
\end{eqnarray*}
Thus there exists a unique, smooth solution $u_i$, defined on time interval $[0,T_i]$, of the Cauchy problem (\ref{eq:4.1}) with $u_0$ replaced
by $u_{i0}$. Furthermore, from the arguments in Lemma \ref{lm:4.1}, we could obtain that if $i$ is large enough, then there exists a uniform
positive lower bound of $T_i$, denoted by $T$, such that the following holds uniformly with respect to large enough $i$:
\begin{eqnarray*}
\sup_{t\in[0,T]} ||\nabla u_i(t)||_{H^{k-1}}\leqslant
C(T,||u_{0x}||_{H^{k-1}}).
\end{eqnarray*}
Hence from Lemma \ref{lm:4.2} we deduce
\begin{eqnarray}\label{eq:www}
\sup_{t\in[0,T]} ||Du_i(t)||_{W^{k-1,2}}\leqslant C(T,||u_{0x}||_{W^{k-1,2}}),
\end{eqnarray}
and by (\ref{eq:4.1}) we have
$${du_i\over dt}\in L^2([0,T],W^{k-3,2}(S^1,N)).$$
By Sobolev theorem, it is easy to see that $u_i\in C^{0,{1\over2}}([0,T],W^{k-3,2}(S^1,N)).$

Interpolating the spaces $L^\infty([0,T],W^{k,2}(S^1,N)) $ and $C^{0,{1\over2}}([0,T],W^{k-3,2}(S^1,N))$
yields that
\begin{equation}\label{eq:vv}
u_i\in C^{0,\gamma}([0,T],W^{k-6\gamma,2}(S^1,N))\quad\text{for}\quad \gamma\in(0,{1\over2}).
\end{equation}
Therefore when letting $\gamma$ small while using Rellich's theorem and the Ascoli-Arzela theorem,
 from (\ref{eq:www}) and (\ref{eq:vv}) we obtain that there exists $$u\in
L^\infty([0,T],W^{k,2}(S^1,N))\cap C([0,T],W^{k-1,2}(S^1,N))$$ such that
\begin{eqnarray*}
u_i&\rightarrow& u\quad[\text{weakly}^*]\quad\text{in}\quad
L^\infty([0,T],W^{k,2}(S^1,N)),\nn\\
u_i&\rightarrow& u\quad\text{in}\quad C([0,T],W^{k-1,2}(S^1,N))
\end{eqnarray*}
upon extracting a subsequence and re-indexing if necessary.

It remains to verify that $u$ is a strong solution to (\ref{eq:4.1}).
We need to check that for any $v\in C^\infty(S^1\times[0,T],\mathbb{R}^n)$ there holds
\begin{eqnarray*}
\int_0^T\int_{S^1}\la u_t,v\ra dx dt&=&\al\int_0^T\int_{S^1}\la J_u\nabla_xu_x,v\ra dx dt\\
&&{}+\beta\left(\int_0^T\int_{S^1}\la \nabla^2_xu_x,v\ra dx dt +
{1\over2}\int_0^T\int_{S^1}\la R_u(u_x,J_uu_x)J_uu_x,v\ra dx dt\right).
\end{eqnarray*}

First we always have that for each $u_i$
\begin{eqnarray*}
\int_0^T\int_{S^1}\la u_{it},v\ra dx dt&=&\al\int_0^T\int_{S^1}\la J_{u_i}\nabla_xu_{ix},v\ra dx dt\\
&&{}+\beta\left(\int_0^T\int_{S^1}\la \nabla^2_xu_{ix},v\ra dx dt +
 {1\over2}\int_0^T\int_{S^1}\la R_{u_i}(u_{ix},J_{u_i}u_{ix})J_{u_i}u_{ix},v\ra dx dt\right).
\end{eqnarray*}
For each $y\in N\subset\mathbb{R}^n$, let $P(y)$ be the orthogonal
projection from $\mathbb{R}^n$ onto $T_yN$, we have
\begin{eqnarray}
\nabla_x u_x&=&P(u)u_{xx},\nn\\
\nabla^2_xu_x&=&P(u)(D(P(u)u_{xx}))\nn\\
&=&P(u)(P(u))_xu_{xx}+P(u)u_{xxx},\label{eq:4.23}\\
\nabla^2_xu_{ix}&=&P(u_i)(P(u_i))_xu_{ixx}+P(u_i)u_{ixxx}.\label{eq:4.24}
\end{eqnarray}

Hence we have
\begin{eqnarray}\label{eq:4.22}
&&{}\int_0^T\int_{S^1}|\la J_u\nabla_xu_x,v\ra- \la J_{u_i}\nabla_xu_{ix},v\ra| dx dt\nn\\
&\leqslant& \int_0^T\int_{S^1}|(J_u-J_{u_i})P(u)u_{xx},v\ra| dx dt+
\int_0^T\int_{S^1}|J_{u_i}(P(u)-P(u_i))u_{xx},v\ra| dx dt\nn\\
&&{}+\int_0^T\int_{S^1}|J_{u_i}P(u_i)(u_{xx}-u_{ixx}),v\ra| dx dt;
\end{eqnarray}
and
\begin{eqnarray}\label{eq:4.26}
&&{}\int_0^T\int_{S^1}|\la\nabla^2_xu_x,v\ra- \la\nabla^2_xu_{ix},v\ra| dx dt\nn\\
&\leqslant&\int_0^T\int_{S^1}|\la\big( P(u)-P(u_i)\big)u_{xxx}, v\ra|dx dt+\int_0^T\int_{S^1}|\la P(u_i)\big(u_{xxx}-u_{ixxx}\big), v\ra|dx dt\nn\\
&&{}+ \int_0^T\int_{S^1}|\la \big(P(u)(P(u))_x-P(u_i)(P(u_i))\big)_xu_{xx},v\ra|dx dt\nn\\
&&{}+ \int_0^T\int_{S^1}|\la P(u_i)(P(u_i))_x\big(u_{xx}-u_{ixx}\big),v\ra|dx dt.
\end{eqnarray}
Moreover,
\begin{eqnarray}\label{eq:4.25}
&&{}\int_0^T\int_{S^1}|\la R_u(u_x,J_uu_x)J_uu_x,v\ra-\la R_{u_i}(u_{ix},J_{u_i}u_{ix})J_{u_i}u_{ix},v\ra| dx dt\nn\\
&\leqslant&\int_0^T\int_{S^1}| R_u(u_x,J_uu_x)J_uu_x-R_{u_i}(u_{ix},J_{u_i}u_{ix})J_{u_i}u_{ix}||v|dx dt.
\end{eqnarray}
Since $N$ is compact, it is obviously that
$$ ||R(\cdot)||_{L^\infty(N)}<\infty\quad\text{and}\quad||P(\cdot)D(P(\cdot))||_{L^\infty(N)}<\infty.$$
Hence we obtain that each term on the right hand side of (\ref{eq:4.22})$-$(\ref{eq:4.25}) converges zero as $i$ goes
to infinity. This implies that
\begin{eqnarray*}
&&{}\lim_{i\rightarrow \infty}\int_0^T\int_{S^1}\la J_{u_i}\nabla_xu_{ix},v\ra dx dt=\int_0^T\int_{S^1}\la J_u\nabla_xu_x,v\ra dx dt;\\
&&{}\lim_{i\rightarrow \infty}\int_0^T\int_{S^1}\la \nabla^2_xu_{ix},v\ra dx dt=\int_0^T\int_{S^1}\la \nabla^2_xu_x,v\ra dx dt;\\
&&{}\lim_{i\rightarrow \infty}\int_0^T\int_{S^1}\la R_{u_i}(u_{ix},J_{u_i}u_{ix})J_{u_i}u_{ix},v\ra dx dt=\int_0^T\int_{S^1}\la R_u(u_x,J_uu_x)J_uu_x,v\ra dx dt.
\end{eqnarray*}
On the other hand, we also have
\begin{eqnarray*}
\lim_{i\rightarrow \infty}\int_0^T\int_{S^1}\la u_{it},v\ra dx dt=-\int _0^T\int_{S^1}\la u,v_t\ra dx dt
+\int_{S^1}(\la u(T),v(T)\ra-\la u_0,v(0)\ra) dx dt.
\end{eqnarray*}
Thus, from the above equalities we have
\begin{eqnarray}\label{eq:4.27}
&&{}\al\int_0^T\int_{S^1}\la J_u\nb_xu_x,v\ra dx dt+\beta\left(
\int_0^T\int_{S^1}\la \nabla^2_xu_x,v\ra dx dt +
{1\over2}\int_0^T\int_{S^1}\la R(u_x,Ju_x)Ju_x,v\ra dx dt\right)\nn\\
&=&-\int _0^T\int_{S^1}\la u,v_t\ra dx dt+\int_{S^1}(\la u(T),v(T)\ra-\la u_0,v(0)\ra) dx dt.
\end{eqnarray}
Note that $\nabla^2_x u_x\in L^2(S^1\times[0,T], \mathbb{R}^n)$,
thus (\ref{eq:4.27}) implies $u_t\in
L^2(S^1\times[0,T],\mathbb{R}^n)$. Therefore for any smooth function
$v$ we always have
\begin{eqnarray*}
\int_0^T\int_{S^1}\la u_t,v\ra dx dt&=&\al\int_0^T\int_{S^1}\la J_u\nb_xu_x,v\ra dx dt\nn\\
&&{}+\beta\left(\int_0^T\int_{S^1}\la \nabla^2_xu_x,v\ra dx dt
 + {1\over2}\int_0^T\int_{S^1}\la R(u_x,Ju_x)Ju_x,v\ra dx dt\right),
\end{eqnarray*}
which means that $u$ is a strong solution of (\ref{eq:4.1}).

It is easy to see that if $N$ is noncompact manifold with
bounded geometry and the domain is $S^1$, we could find a compact
subset of $N$, denoted by $\Omega$, such that $u_0(S^1)\subset
\Omega\subset \mathbb{R}^n$. Then we could repeat the same process
as in the case $N$ is compact (also see \cite{Mc}), then we could
obtain the same results.

Hence we complete the proof of Lemma \ref{thm:1.1}. Moreover, when
$k\geqslant4$, Theorem \ref{thm:1.3} bellow asserts that the
solutions are unique.
\hspace*{\fill}$\Box$\\

\noindent Now we prove Theorem \ref{thm:1.3}
and show the uniqueness of the solutions.\\

\noindent{\bf\textit{Proof of Theorem \ref{thm:1.3}}}.
Without loss of generality, we may assume that $N$ is compact, since
$u(x, t)\in L^\infty([0, T], H^4(S^1, N))$ implies that $\{u(x, t): (x, t)
\in S^1\times[0, T]\}\subset\subset N$. We regard $N$ as a submanifold of
$\mathbb{R}^n$. Let $u$, $v:S^1\times[0,T]\rightarrow N\subset\mathbb{R}^n$
be two solutions of (\ref{eq:4.1}) such that $u(x,0)=v(x,0)=u_0$ and
$u,v\in L^\infty([0,T],W^{k,2}(S^1,N))$ for $k\geqslant4$. Let $w=u-v$ and it
makes sense as a $\mathbb{R}^n$-valued function. It is worthy to point out
that both the curvature tensor $R$ and the complex structure $J$ here should
be regarded as operators on $\mathbb{R}^n$, such that $R(u)(u_x,J_uu_x)J_uu_x-R(v)(v_x,J_vv_x)J_vv_x$ makes sense
in $\mathbb{R}^n$.

By the discussion, in Lemma \ref{lm:4.2}, for the solution $u$ of (\ref{eq:1.15}), we have
\begin{eqnarray*}
u_t= \al J_u P(u)u_{xx}+\beta\left(
P(u)u_{xxx}+P(u)(P(u))_xu_{xx}+{1\over2}R(u)(u_x,J_uu_x)J_uu_x\right).
\end{eqnarray*}
Hence we have
\begin{eqnarray}\label{eq:4.45}
w_t&=&\al \Big(J_uP(u)u_{xx}-J_vP(v)v_{xx}\Big)\nn\\
&&{}+\beta \Bigg(P(u)w_{xxx}+[P(u)-P(v)]v_{xxx}\nn\\
&&+P(u)(P(u))_xw_{xx}+[P(u)(P(u))_x-P(v)(P(v))_x]v_{xx}\nn\\
&&{}+{1\over2}(R(u)(u_x,J_uu_x)J_uu_x-R(v)(v_x,J_vv_x)J_vv_x)\Bigg).
\end{eqnarray}

We could show that there exists a constant $C$ which only depends on $N$,
$||u||_{W^{4,2}}$, $||v||_{W^{4,2}}$, $\al$, and $\beta$  such that
\begin{eqnarray}\label{eq:4.29}
{d\over dt}||w||^2_{W^{1,2}}\leqslant C(N,||u||_{W^{4,2}},||v||_{W^{4,2}},\al,\beta) ||w||^2_{W^{1,2}}.
\end{eqnarray}
Then by Gronwall's inequality we could obtain that $w\equiv0$ which yields
 the uniqueness of the solution. We
omit the details about the proof of (\ref{eq:4.29}) since the process is almost
the same with that in \cite{SW}.

Thus it suffices to show that$\nabla^{k-1}_xu_x\in C([0,T];L^2(S^1,TN))$ for $k\geqslant4$. In the proof
of Theorem \ref{thm:1.1} we have seen
that the solution $u\in L^\infty([0,T],H^{k}(S^1,N))\cap C([0,T],H^{k-1}(S^1,N)),$  thus by the discussion about
(\ref{eq:ww}), (\ref{eq:www}) and the equation of Schr\"odinger-Airy flow, we could
easily get that
$${d\over dt}||\nabla^{k-1}_xu_x||^2_{L^2}\leqslant C,$$
which implies that
$$||\nabla_x^{k-1}u_x(t,x)||^2_{L^2(S^1,TN)}\leqslant ||\nabla^{k-1}_xu_x(0,x)||^2_{L^2(S^1,TN)}+Ct.$$

Hence we obtain
$$\lim_{t\rightarrow 0}\sup||\nabla_x^{k-1}u_x(t,x)||^2_{L^2(S^1,TN)}\leqslant ||\nabla^{k-1}_xu_x(0,x)||^2_{L^2(S^1,TN)}.$$

On the other hand, $u\in L^\infty([0,T],H^{k}(S^1,N))\cap C([0,T],H^{k-1}(S^1,N))$ implies that, with respect to $t$,
$\nabla^{k-1}_xu_x(t,x)$ is weakly continuous in $L^2(S^1,TN)$,
we have
$$||\nabla^{k-1}_xu_x(0,x)||^2_{L^2(S^1,TN)}\leqslant \lim_{t\rightarrow0}\inf||\nabla^{k-1}_xu_x(t,x)||^2_{L^2(S^1,TN)}.$$


Thus,
$$\lim_{t\rightarrow0} ||\nabla^{k-1}_xu_x(t,x)||^2_{L^2}=||\nabla^{k-1}_xu_x(0,x)||^2_{L^2},$$
which implies that $\nabla^{k-1}_xu_x(t,x)$ is continuous in $L^2(S^1,TN)$ at $t=0$. Now by the uniqueness of $u(t,x)$,
we get that $\nabla_x^{k-1}u_x(t,x)$ is continuous at each $t\in[0,T]$, i.e. $u\in C([0,T],H^k(S^1,N))$ for all
$k\geqslant4$. However, if $k\leqslant3$, we could not get the continuity of $||u||_{H^k}$ about $t$ on $[0,T]$ without
the uniqueness of $u$. Thus we
complete the proof of Theorem \ref{thm:1.3}. \hspace*{\fill}$\Box$\\

\noindent{\bf\textit{Proof of Theorem \ref{thm:1.4}}}. To
show the existence of the Cauchy problem (\ref{eq:1.15}) with an
initial map $u_0\in H^4(S^1, N)$, we need to consider the following
Cauchy problems:
\begin{equation}\label{eq:4.30}
\left\{
\begin{aligned}
&u_t=\al J\nb_xu_x+\beta\left(
\nabla_x^2u_x+{1\over2}R(u_x,J_uu_x)J_uu_x\right),\quad x\in S^1;\\
&u(x,0)=u_0^i(x).
\end{aligned}\right.
\end{equation}
Here $u_0^i\in C^\infty(S^1, N)$ and $\|u_0^i-u_0\|_{H^4}\rightarrow
0$. By Lemma 3.1 we know that for each $i$ and any  $k\geq
5$, (\ref{eq:4.30}) admits a local solution $u^i\in
L^\infty([0, T^{\max}_i), H^k(S^1, N))$, where
$T_i^{\max}=T_i^{\max}(S^1, \|u_0^i\|_{H^5})$ is the maximal
existence interval of $u^i$.

As $N$ is not compact we let $\Omega_i\triangleq \{p\in
N:\text{dist}_N(p,u_0^i(S^1))<1\}$, which is an open subset of $N$
with compact closure $\bar{\Omega}_i$. Denote
$$\Omega_\infty\triangleq \{p\in
N:\text{dist}_N(p,u_0(S^1))<1\}\quad\mbox{and}\quad\Omega_0
\triangleq \{p\in N: \text{dist}_N(p,\Omega_\infty)<1\}.$$ Since
$\|u_0^i-u_0\|_{H^4}\rightarrow 0$, then $\Omega_i \subset\subset
\Omega_0$ as $i$ is large enough. Let
 $$T'_i=\sup\{t>0: u^i(S^1,t)\subset\Omega_i\}.$$
By the same argument as in Lemma 3.1 we can show that there holds
true for all $t\in[0,T_i]$
\begin{eqnarray*}
{d\over dt}||u_x^i||^2_{H^2}\leqslant C(\Omega_0)\sum_{l=2}^4
||u_x^i||^{2l}_{H^2}.
\end{eqnarray*}
If we let $f^i(t)=||u^i_x||^2_{H^2}+1$, then we have
\begin{eqnarray*}
{df^i\over dt}\leqslant C(\Omega_0)(f^i)^4,\quad f^i(0)=||u^i_{0x}||^2+1.
\end{eqnarray*}
It follows from the above differential inequality that there holds true
$$f^i(t)\leqslant \left(\frac{(f^i(0))^3}{1-3(f^i(0))^3C(\Omega_0)t}\right)^{\frac{1}{3}},$$
as
$$t<\frac{1}{3(f^i(0))^3C(\Omega_0)}.$$
Then, there exists constants
$$T_0^i=T_0^i(\Omega_0, \|u^i_{0x}\|_{H^2})=\frac{1}{4(f^i(0))^3C(\Omega_0)}>0,$$
$$C^i_0= 4^{\frac{1}{3}}f^i(0)>0$$ and
$C^i_{k-2}=C_{k-2}(\|u^i_{0x}\|_{H^{k}})>0$ such that
\begin{eqnarray*}
||u^i_{x}||_{H^2}\leqslant C^i_0,\quad t\in[0,\text{min}(T^i_0,T'_i)].
\end{eqnarray*}
and for $k\geq 3$
\begin{eqnarray*}
||u^i_{x}||_{H^k}\leqslant C^i_{k-2},\quad
t\in[0,\text{min}(T^i_0,T'_i)].
\end{eqnarray*}
Since $\|u_0^i-u_0\|_{H^4}\rightarrow 0$, when $i$ is large enough we have
$$T_0=\frac{1}{4(\|u_{0x}\|_{H^2}^2+1 +\delta_0)C(\Omega_0)}<T^i_0,$$
where $\delta_0$ is a small positive number. It is easy to see that,
as $i$ is large enough,

\begin{equation}\label{eq:3.55*}C^i_0\leq
C_0(\|u_{0x}\|_{H^2})+\delta_0\quad\mbox{and}\quad C^i_1\leq
C_1(\|u_{0x}\|_{H^3})+\delta_0.
\end{equation}

In fact, we always have $T_i^{\max}> \text{min}(T_0,T'_i)$ when $i$
is large enough. Otherwise, by Lemma 3.1 we can find a time-local solution $u_1$ of
(\ref{eq:1.15}) and $u_1$ satisfies the initial value condition
$$u_1(x,T_i^{\max}-\epsilon)=u(x,T_i^{\max}-\epsilon),$$
where $0<\epsilon<T_i^{\max}$ is a small number. Then by the local existence
theorem, $u_1$ exists on the time interval
$(T_i^{\max}-\epsilon,T_i^{\max}-\epsilon+\eta)$ for some constant $\eta>0$. The
uniform bounds on $||u_x||_{H^2}$ and $||\nabla_x^m u_x ||_{L^2}$
(for all $m>2$) implies that $\eta$ is independent of $\epsilon$.
Thus, by choosing $\epsilon$ sufficiently small, we have
$$T_i^e=T_i^{\max}-\epsilon+\eta>T_i^{\max}.$$
By the uniqueness result, we have that $u_1(x,t)=u(x,t)$ for all
$t\in[T_i^{\max}-\epsilon,T_i^e)$. Thus we get a solution of the Cauchy problem
(\ref{eq:1.15}) on the time interval $[0,T_e)$,
which contradicts the maximality of $T_i^{\max}$.

Now we need to show that $T'_i$ have a uniform lower bound as $i$ is large enough.
For each large enough $i$, if $T'_i\geqslant T_0$ we obtain the lower bound. Otherwise,
by the same argument as in Lemma 3.1 we have
$$T'_i\geqslant \frac{1}{\mathcal{M}_i}$$
where $$\mathcal{M}_i= \sup_{[0, \text{min}(T_0,T'_i)]}
||u_t^i||_{L^\infty}\leqslant C \sup_{[0, \text{min}(T_0,T'_i)]}
||u^i_t||^a_{H^1}\sup_{[0, \text{min}(T_0,T'_i)]}
||u^i_t||^{1-a}_{L^2}\equiv M_i.$$ It should be pointed out that to
derive the estimates $L^\infty$ estimates on $||u^i_t(s)||_{L^2}$ and
$||u^i_t(s)||_{H^1}$ we need only to have $u^i\in L^\infty([0,
\text{min}(T_0,T'_i)], H^4(S^1, N))$, since the equation of KdV flow
is a third-order dispersive equation. It is not difficult to see
from (\ref{eq:3.55*}) that there exists a positive constant
$M(\Omega_0, ||u_{0x}||_{H^3})$ such that, as $i$ is large enough,
 $$M_i(\Omega_0, ||u^i_{0x}||_{H^3})\leqslant M(\Omega_0, ||u_{0x}||_{H^3}),$$
since $\|u_0^i-u_0\|_{H^4}\rightarrow 0$.

Let $T^*=\text{min}(T_0, \frac{1}{M})$. As $i$ is large enough, we always have
$u^i \in L^\infty([0, T^*], H^4(S^1, N))$. By letting $i\rightarrow \infty$ and
taking the same arguments as in Theorem 1.1, we know there exists $u \in L^\infty([0, T^*], H^4(S^1, N))$
such that
 $$u^i \rightarrow u\quad[\text{weakly}^*]\quad\text{in}\quad L^\infty([0,T^*],H^4(S^1,N))$$
and $u$ is a local solution to (\ref{eq:4.1}). Theorem \ref{thm:1.3} tells us
that the local solution is unique and the local solution is continuous with
respect to $t$, i.e., $u\in C([0, T^*], H^4(S^1, N))$. Thus, we complete the proof of the
theorem. \hspace*{\fill}$\Box$

\section{Conservation Laws}
In this section we show the conservation laws $E_1(u)$, $E_2(u)$, $E_3(u)$ and
the semi-conservation law $E_4(u)$
introduced in the first section with some special assumptions about $N$.
These will help us to obtain the global existence of the KdV geometric flow in the next section.

First, we recall that a complex manifold $N$ of real dimension $2n$ and
integrable complex structure $J$ is said to be K\"ahler if it possesses a
Riemannian metric $h$ for which $J$ is an isometry, as well as a symplectic
form $\omega$ satisfying the compatibility condition $\omega(X, Y ) = g(JX, Y )$
for all tangent vector fields $X$, $Y\in \Gamma(TN)$ which denotes the space
of smooth sections on $TN$.

Now we derive some conservation laws of the new geometric flow
(\ref{eq:1.15}) on a locally Hermitian symmetric space. The computational process about the KdV part is the same with that in \cite{SW}
thus we omit many details and use the results in \cite{SW} directly.
\begin{lem}\label{lm:3.1}
Assume $N$ is a locally Hermitian symmetric space. If $u:S^1\times (0,T)\rightarrow N$ is a smooth solution of
the Cauchy problem of the Hirota geometric flow (\ref{eq:1.15}), then
$${dE_1\over dt}={1\over 2}{d\over dt}\int|u_x|^2dx=0,$$ in other
words, $E_1(u)=E_1(u_0)$ for all $t\in(0,T)$.
\end{lem}
{\textit{Proof}. With the assumption of $N$, we have $\nabla R=0$.
Similar with the computation before, we get
\begin{eqnarray}\label{eq:4.0}
&&{dE_1\over dt}={1\over 2}{d\over dt}\int|u_x|^2dx=-\int\la\nabla_x u_x,u_t\ra dx\nn\\
&=&-\alpha\int\la \nb_x u_x,J\nb_xu_x\ra dx-\beta\left(\int\la\nabla_x u_x,
\nabla_x^2 u_x\ra dx+{1\over2}\int\la \nabla_x u_x,
R(u_x,Ju_x)Ju_x\ra dx\right).\nn\\
\end{eqnarray}

Obviously, by the antisymmetry of $J$, the first part on the right of (\ref{eq:4.0}) vanishes.
The second part vanishes too, since in \cite{SW} we have proved that $E_1$ is
  preserved by KdV geometric flow with the same assumptions. Precisely, we have
\begin{eqnarray*}
&&\beta\left(\int\la\nabla_x u_x,
\nabla_x^2 u_x\ra dx+{1\over2}\int\la \nabla_x u_x,
R(u_x,Ju_x)Ju_x\ra dx\right)\nn\\
&=&{\beta\over2}\int\nabla_x(\la\nabla_x u_x, \nabla_x u_x\ra)
dx+{\beta\over8}\int\nabla_x (\la u_x, R(u_x,Ju_x)Ju_x\ra) dx\\
&&{}-{\beta\over8}\int \la u_x, (\nabla_x R)(u_x,Ju_x)Ju_x\ra dx\\
&=&0.
\end{eqnarray*}
The last equality holds since $\nabla_x R=0$. Hence we have
$${dE_1\over dt}=0.$$
This completes the
proof.\hspace*{\fill}$\Box$\\

\begin{lem}\label{lm:3.2}
Assume $N$ is a is a locally Hermitian symmetric space. If $u:S^1\times (0,T)\rightarrow N$ is a smooth solution of
the Cauchy problem of the geometric Schr\"odinger-Airy flow (\ref{eq:1.15}), the pseudo-helicity of $u$
$$E_2(u)\equiv\int\la\nabla_x u_x,Ju_x\ra dx$$
is preserved too, i.e.
$${dE_2\over dt}=0.$$
\end{lem}
\textit{Proof}. Differentiating $E_2$ with respect to $t$, we have
\begin{eqnarray*}
{dE_2\over dt}&=&\int \la \nabla_t\nabla_x u_x, Ju_x\ra dx+\int\la
\nabla_x u_x, J\nabla_t u_x\ra dx\\
&=&\int\la\nabla_x^2u_t ,Ju_x\ra dx+\int\la R(u_t,u_x)u_x, Ju_x\ra
dx+\int\la \nabla_x u_x, J\nabla_x u_t\ra dx\\
&=&2\int\la u_t,J\nabla_x^2u_x\ra dx-\int \la u_t,R(u_x,Ju_x)u_x\ra dx.
\end{eqnarray*}

 Substituting (\ref{eq:1.15}) yields,
\begin{eqnarray*}
{dE_2\over dt}
&=&\al\left(2\int\la J\nb_xu_x,J\nabla_x^2u_x\ra dx-\int \la J\nb_xu_x,R(u_x,Ju_x)u_x\ra dx\right)\\
&&{}+\beta\left(\int \la R(u_x,Ju_x)Ju_x, J\nabla_x^2 u_x\ra dx-\int\la \nabla_x^2u_x,
R(u_x,Ju_x)u_x\ra dx\right.\\
&&{}-\left.{1\over2}\int\la R(u_x,Ju_x)Ju_x,R(u_x,Ju_x)u_x\ra dx\right).
\end{eqnarray*}

We have known that KdV flow on any K\"ahler manifold preserves $E_2$ (see \cite{SW}).
Thus, if $N$ is a locally Hermitian symmetric space, the same computation yields
\begin{eqnarray*}
{dE_2\over dt}
&=&\al\left(2\int\la J\nb_xu_x,J\nabla_x^2u_x\ra dx-\int \la J\nb_xu_x,R(u_x,Ju_x)u_x\ra dx\right)\\
&=&-\al \int \la \nb_xu_x,R(u_x,Ju_x)Ju_x\ra dx\\
&=&-{\al\over 4}\int\nabla_x (\la u_x, R(u_x,Ju_x)Ju_x\ra) dx=0.
\end{eqnarray*}
This completes the proof. \hspace*{\fill}$\Box$\\

 Now we prove the third conservation law $E_3$ with some special conditions about $N$.

 To begin with, we compute ${dE_3\over dt}$ as before where
$$E_3(u)\equiv\int|\nabla_x u_x|^2dx-{1\over4}\int\langle
u_x,R(u_x,Ju_x)Ju_x\rangle dx.$$

 For the first term of $E_3$, we have
\begin{eqnarray*}
{d\over dt}\int| \nabla_x u_x|^2 dx
&=&2\int \la \nabla_x^3 u_x,u_t\ra dx+2\int\la u_t, R(\nabla_x
u_x,u_x)u_x\ra dx.
\end{eqnarray*}
Substituting $u_t$ we get
\begin{eqnarray}\label{eq:3.2}
{d\over dt}\int| \nabla_x u_x|^2 dx
&=&\al\left(2\int \la \nabla_x^3 u_x,J\nb_xu_x\ra dx+2\int\la J\nb_xu_x, R(\nabla_x
u_x,u_x)u_x\ra dx\right)\nn\\
&&{}+\beta\left(2\int \la
\nabla_x^3u_x,\nabla_x^2u_x\ra dx+2\int\la\nabla_x^2u_x,R(\nabla_x
u_x,u_x)u_x\ra dx\right.\nn\\
&&{}+\int\la\nabla_x^3u_x,R(u_x,Ju_x)Ju_x\ra dx\nn\\
&&{}+\left.\int\la R(\nabla_x u_x,u_x)u_x\ra, R(u_x,Ju_x)Ju_x\ra dx\right).
\end{eqnarray}

For the $\al$-part of (\ref{eq:3.2}), using integration by parts and the property of $J$ yields
\begin{eqnarray*}
&&{}\al\left(2\int \la \nabla_x^3 u_x,J\nb_xu_x\ra dx+2\int\la J\nb_xu_x, R(\nabla_x
u_x,u_x)u_x\ra dx\right)\\
&=&-\al\left(2\int \la \nabla_x^2 u_x,J\nb^2_xu_x\ra dx+2\int\la \nb_xu_x, R(\nabla_x
u_x,u_x)Ju_x\ra dx\right)\\
&=&-2\al\int\la \nb_xu_x, R(\nabla_x
u_x,u_x)Ju_x\ra dx.
\end{eqnarray*}

For the $\beta$-part of (\ref{eq:3.2}), we proceed as that in (\cite{SW}). Then we obtain that
\begin{eqnarray}\label{eq:3.8}
{d\over dt}\int| \nabla_x u_x|^2 dx&=&
-2\al\int\la \nb_xu_x, R(\nabla_x
u_x,u_x)Ju_x\ra dx\nn\\
&&{}+\beta\left(3\int\la\nabla_x u_x,R(\nabla_x
u_x,Ju_x)J\nabla_x u_x\ra dx\right.\nn\\
&&{}+\left.\int\la R(\nabla_x u_x,u_x)u_x\ra, R(u_x,Ju_x)Ju_x\ra dx\right).
\end{eqnarray}

Now we turn to the second term of $E_3$. Differentiating it with respect to $t$ and
using the symmetry of $R$ yields
\begin{eqnarray}\label{eq:3.9}
&&{}-{1\over4}{d\over dt}\int\la u_x,R(u_x,Ju_x)Ju_x\ra dx
=-\int\la\nabla_xu_t, R(u_x,Ju_x)Ju_x\ra dx\nn\\
&=&-\al\int\la J\nabla_x^2u_x, R(u_x,Ju_x)Ju_x\ra dx\nn\\
&&{}-\beta\left(\int\la\nabla_x^3u_x,R(u_x,Ju_x)Ju_x\ra dx+
{1\over2}\int\la\nabla_x (R(u_x,Ju_x)Ju_x),R(u_x,Ju_x)Ju_x\ra dx\right)\nn\\
&=&-\al\int\la \nabla_x^2u_x, R(u_x,Ju_x)u_x\ra dx
-\beta\int\la\nabla_x^3u_x,R(u_x,Ju_x)Ju_x\ra dx\nn\\
&=&2\al\int\la \nabla_xu_x, R(\nb_xu_x,u_x)Ju_x\ra dx
-3\beta\int\la\nabla_x
u_x,R(\nabla_x u_x,Ju_x)J\nabla_x u_x\ra dx.
\end{eqnarray}

Combining (\ref{eq:3.8}) and (\ref{eq:3.9}) we get that
\begin{eqnarray}\label{eq:3.10}
 {dE_3\over dt}&=&{d\over dt}\left(\int| \nabla_x u_x|^2 dx
 -{1\over4}\int\la u_x,R(u_x,Ju_x)Ju_x\ra dx\right)\nn\\
 &=&\beta\int\la R(\nabla_x u_x,u_x)u_x, R(u_x,Ju_x)Ju_x\ra dx.
\end{eqnarray}

\begin{lem}\label{lem:3.3*}
Let $N$ is a locally Hermitian symmetric space. Then, for a smooth
solution $u:S^1\times (0,T)\rightarrow N$ to Hirota geometric flow (\ref{eq:1.15}) on $N$, we have
$${dE_3\over dt}=\beta\int\la R(\nabla_x u_x,u_x)u_x, R(u_x,Ju_x)Ju_x\ra dx.$$
\end{lem}

Thus we obtain that $E_3(u)$ would be conserved under some additional conditions about $N$.
More precisely, the following lemma helps us get the
conservation law.

\begin{lem}\label{lem:3.4*}(\cite{SW})
Assume $(N, J, h)$ is one of the following three kinds of manifolds:
K\"{a}hler manifolds with constant holomorphic sectional curvature
$K$, complex Grassmannians $G_{p,q}(\mathbb{C})$ and a complex
hyperbolic Grassmannians $D_{m,l}(\mathbb{C})$. Then there always
holds true
 $$\la R(Y, X)X, R(X, JX)JX\ra\equiv 0,$$
for any tangent vector fields $X$ and $Y$ on $N$.
\end{lem}

\begin{coro}\label{lm:3.3}
If $N=M_1\times M_2\times \cdots \times M_n$ is a product manifold
where $(M_i, h^i)$ ($i=1,2,\cdots,n$) is a locally Hermitian
symmetric space satisfying $h^i(R^i(Y, X)X, R^i(X, JX)JX)\equiv 0$
where $R^i$ is the Riemann curvature on $M_i$, then for a smooth
solution $u:S^1\times (0,T)\rightarrow N$ to
Hirota geometric flow (\ref{eq:1.15}) on $N$,  $E_3(u)$ is
preserved, i.e.,
 $${dE_3\over dt}={d\over dt}\left(\int| \nabla_x u_x|^2 dx
 -{1\over4}\int\la u_x,R(u_x,Ju_x)Ju_x\ra dx\right)=0.$$
\end{coro}

 The proof is easy and standard. We omit the detail. \\

Next we prove the semi-conservation law about $E_4$, where
\begin{eqnarray*}
E_4(u)&=&2\int|\nabla_x^2u_x|^2dx-3\int\langle\nabla_x u_x,R(\nabla_x u_x,u_x)u_x\rangle dx\\
&&{}-5\int\langle\nabla_x u_x,R(\nabla_x u_x ,Ju_x)Ju_x\rangle dx.
\end{eqnarray*}
\begin{lem}\label{lm:3.4}
With the same assumptions as in Lemma \ref{lm:3.1}, we have
$${dE_4\over dt}\leqslant C(E_4+1),$$ where $C$ is a constant depends on
$N$, $E_1(u_0)$ and $||\nabla_x u_x||_{L^2}$.
\end{lem}
\textit{Proof}. For simplicity, we still denote
\begin{eqnarray*}
E_4(u)\triangleq A_1F_1+A_2F_2+A_3F_3,
\end{eqnarray*}
as before, where $A_1=2$, $A_2=-3$, $A_3=-5$,
$$F_1\triangleq\int|\nabla_x^2u_x|^2dx,$$
$$F_2\triangleq\int\langle \nabla_x u_x,R(\nabla_x u_x,u_x)u_x\rangle dx,$$
and
$$F_3\triangleq\int\langle\nabla_x u_x,R(\nabla_x u_x ,Ju_x)Ju_x\rangle dx$$
respectively.

We first consider $F_1$.  Differentiating it with respect to $t$, we have
\begin{eqnarray}\label{eq:3.13}
{dF_1\over dt}&=&{d\over dt}\int|\nabla_x^2u_x|^2dx\nn\\
&=&-2\int\la u_t,\nabla_x^5u_x\ra dx-2\int\la u_t,R(\nabla_x^3u_x,u_x)u_x\ra dx\nn\\
&&{}+2\int\la u_t,R(\nabla_x^2u_x,\nabla_x u_x)u_x\ra dx.
\end{eqnarray}

After substituting (\ref{eq:1.15}) into above, we get two parts in the equality, i.e. the
Schr\"odinger part and the KdV part. Here we mainly deal with the first part. For the KdV part,
we use the results in \cite{SW}. Then we have

\begin{eqnarray}\label{F1}
&&{}{dF_1\over dt}={d\over dt}\int|\nabla_x^2u_x|^2dx=\alpha I_{11}+\beta I_{12},
\end{eqnarray}
where
\begin{eqnarray}
I_{11}&=&-2\int\la J\nb_xu_x,\nabla_x^5u_x\ra dx-2\int\la J\nb_xu_x,R(\nabla_x^3u_x,u_x)u_x\ra dx\nn\\
&&{}+2\int\la J\nb_xu_x,R(\nabla_x^2u_x,\nabla_x u_x)u_x\ra dx;\nn
\end{eqnarray}
 and $I_{12}$ is the KdV part:
\begin{eqnarray}
I_{12}&=&15\int\la \nabla_x^2u_x,R(\nabla_x^2 u_x,Ju_x)J\nabla_x u_x\ra dx+9\int\la \nabla_x^2u_x,R(\nabla_x^2 u_x,\nabla_x u_x)u_x\ra dx\nn\\
&&{}+2\int\la R(u_x,Ju_x)Ju_x,R(\nabla_x^2 u_x,\nabla_x u_x)u_x\ra dx\nn\\
&&{}+2\int\la R(\nabla_x u_x,Ju_x)Ju_x,R(\nabla_x^2 u_x,u_x)u_x\ra dx\nn\\
&&{}+\int\la R(u_x,Ju_x)Ju_x,R(\nabla_x^2 u_x,u_x)\nabla_x u_x\ra dx\nn\\
&&{}+\int\la R(u_x,Ju_x)J\nabla_x u_x,R(\nabla_x^2 u_x,u_x)u_x\ra dx.
\end{eqnarray}
Obviously, the first term of $I_{11}$ vanishes since
$$\int\la J\nb_xu_x,\nabla_x^5u_x\ra dx=\int\la J\nb_x^3u_x,\nabla_x^3u_x\ra dx=0.$$
For the second term of $I_{11}$, we have
\begin{eqnarray}
&&{}-2\int\la J\nb_xu_x,R(\nabla_x^3u_x,u_x)u_x\ra dx
=2\int\la \nabla_x^3u_x,R(\nb_xu_x,Ju_x)u_x\ra dx\nn\\
&=&-2\int\la \nabla_x^2u_x,R(\nb^2_xu_x,Ju_x)u_x\ra dx
-2\int\la \nabla_x^2u_x,R(\nb_xu_x,J\nb_xu_x)u_x\ra dx\nn\\
&&{}-2\int\la \nabla_x^2u_x,R(\nb_xu_x,Ju_x)\nb_xu_x\ra dx.\nn
\end{eqnarray}
For the third term of $I_{11}$, it is easy to check that
$$2\int\la J\nb_xu_x,R(\nabla_x^2u_x,\nabla_x u_x)u_x\ra dx
=-2\int\la \nb_x^2u_x,R(\nabla_xu_x, Ju_x)\nabla_xu_x\ra dx$$
Hence
\begin{eqnarray}\label{I11}
I_{11}&=&-2\int\la \nabla_x^2u_x,R(\nb^2_xu_x,Ju_x)u_x\ra dx
-2\int\la \nabla_x^2u_x,R(\nb_xu_x,J\nb_xu_x)u_x\ra dx\nn\\
&&{}-4\int\la \nabla_x^2u_x,R(\nb_xu_x,Ju_x)\nb_xu_x\ra dx.
\end{eqnarray}

Next we compute $dF_2\over dt$. From \cite{SW} we have
\begin{eqnarray}\label{eq:3.22}
&&{}{dF_2\over dt}={d\over dt}\int\langle \nabla_x u_x,R(\nabla_x u_x,u_x)u_x\rangle dx\nn\\
&=&2\int\langle u_t,R(\nabla_x^3 u_x,u_x)u_x\rangle dx-2\int\langle u_t,R(\nabla_x^2 u_x,\nabla_x u_x)u_x\rangle dx\nn\\
&&{}+10\int\langle u_t,R(\nabla_x^2 u_x,u_x)\nabla_x u_x\rangle dx
+2\int\langle u_t,R(R(\nabla_x u_x,u_x)u_x,u_x)u_x\rangle dx.\nn
\end{eqnarray}

Thus, substituting (\ref{eq:1.15}) yields
\begin{eqnarray}\label{F2}
&&{dF_2\over dt}={d\over dt}\int\langle \nabla_x u_x,R(\nabla_x u_x,u_x)u_x\rangle dx
=\alpha I_{21}+\beta I_{22},
\end{eqnarray}
where
\begin{eqnarray}
I_{21}&=&2\int\langle J\nb_xu_x,R(\nabla_x^3 u_x,u_x)u_x\rangle dx
-2\int\langle J\nb_xu_x,R(\nabla_x^2 u_x,\nabla_x u_x)u_x\rangle dx\nn\\
&&{}+10\int\langle J\nb_xu_x,R(\nabla_x^2 u_x,u_x)\nabla_x u_x\rangle dx\nn\\
&&{}+2\int\langle J\nb_xu_x,R(R(\nabla_x u_x,u_x)u_x,u_x)u_x\rangle dx;\nn
\end{eqnarray}
and $I_{22}$ is the KdV part
\begin{eqnarray}
I_{22}&=&6\int\la\nabla_x^2u_x,R(\nabla_x^2 u_x,u_x)\nabla_x u_x\ra dx\nn\\
&&{}-2\int\la R(u_x,Ju_x)Ju_x,R(\nabla_x^2 u_x,\nabla_x u_x)u_x\ra dx\nn\\
&&{}+4\int\la R(u_x,Ju_x)Ju_x,R(\nabla_x^2 u_x,u_x)\nabla_x u_x\ra dx\nn\\
&&{}-\int\la R(u_x,Ju_x)J\nabla_x u_x,R(\nabla_x^2 u_x,u_x)u_x\ra dx\nn\\
&&{}-2\int\la R(\nabla_x u_x,Ju_x)Ju_x,R(\nabla_x^2 u_x,u_x)u_x\ra dx\nn\\
&&{}+2\int\langle R(\nabla_x u_x,u_x)u_x,R(\nabla_x^2u_x,u_x)u_x\rangle dx\nn\\
&&{}+\int\langle R(u_x,Ju_x)Ju_x,R(R(\nabla_x u_x,u_x)u_x,u_x)u_x\rangle dx.
\end{eqnarray}
For the first term of $I_{21}$, integrating by parts yields
\begin{eqnarray}
&&{}2\int\langle J\nb_xu_x,R(\nabla_x^3 u_x,u_x)u_x\rangle dx
=-2\int\langle \nb_x^3u_x,R(\nabla_x u_x,Ju_x)u_x\rangle dx\nn\\
&=&2\int\langle \nb_x^2u_x,R(\nabla^2_x u_x,Ju_x)u_x\rangle dx
+2\int\langle \nb_x^2u_x,R(\nabla_x u_x,J\nb_xu_x)u_x\rangle dx\nn\\
&&{}+2\int\langle \nb_x^2u_x,R(\nabla_x u_x,Ju_x)\nb_xu_x\rangle dx.\nn
\end{eqnarray}
Moreover,
$$10\int\langle J\nb_xu_x,R(\nabla_x^2 u_x,u_x)\nabla_x u_x\rangle dx
=-10\int\langle \nb^2_xu_x,R(\nabla_x u_x,J\nabla_x u_x)u_x\rangle dx.$$
Hence we have
\begin{eqnarray}\label{I21}
I_{21}&=&2\int\langle \nb_x^2u_x,R(\nabla^2_x u_x,Ju_x)u_x\rangle dx
-8\int\langle \nb_x^2u_x,R(\nabla_x u_x,J\nb_xu_x)u_x\rangle dx\nn\\
&&{}+4\int\langle \nb_x^2u_x,R(\nabla_x u_x,Ju_x)\nb_xu_x\rangle dx\nn\\
&&{}-2\int\langle \nb_xu_x,R(R(\nabla_x u_x,u_x)u_x,u_x)Ju_x\rangle dx.
\end{eqnarray}

Similarly, we deduce:
\begin{eqnarray}\label{F3}
{dF_3\over dt}&=&{d\over dt}\int\langle \nabla_x u_x,R(\nabla_x u_x,Ju_x)Ju_x\rangle dx\nn\\
&=&2\int\la u_t,R(\nabla_x^3u_x,Ju_x)Ju_x\ra dx
+6\int\la u_t,R(\nabla_x^2u_x,J\nabla_x u_x)Ju_x\ra dx\nn\\
&&{}+2\int\la u_t,R(\nabla_x^2u_x,Ju_x)J\nabla_x u_x\ra dx
+2\int\la u_t,R(\nabla_x u_x,J\nabla_x u_x)J\nabla_x u_x\ra dx\nn\\
&&{}+2\int\la u_t,R\big(R(\nabla_x u_x,Ju_x)Ju_x,u_x\big)u_x\ra dx\nn\\
&=&\alpha I_{31}+\beta I_{32}.
\end{eqnarray}

Here $I_{32}$ is the KdV part:
\begin{eqnarray}\label{eq:3.31}
I_{32}&=&6\int\la\nabla_x^2u_x,R(\nabla_x^2u_x,J\nabla_x u_x)Ju_x\ra dx\nn\\
&&{}-\int\la R(u_x,Ju_x)J\nabla_x u_x,R(\nabla_x^2u_x,Ju_x)Ju_x\ra dx\nn\\
&&{}-2\int\la R(\nabla_x u_x,Ju_x)Ju_x,R(\nabla_x^2u_x,Ju_x)Ju_x\ra dx\nn\\
&&{}+2\int\la R(u_x,Ju_x)Ju_x,R(\nabla_x^2u_x,J\nabla_x u_x)Ju_x\ra dx\nn\\
&&{}+\int\la R(u_x,Ju_x)Ju_x,R(\nabla_x u_x,J\nabla_x u_x)J\nabla_x u_x\ra dx\nn\\
&&{}+2\int\la R(\nabla_x u_x,Ju_x)Ju_x,R\big(\nabla_x^2u_x,u_x\big)u_x\ra dx\nn\\
&&{}+\int\la R(u_x,Ju_x)Ju_x,R\big(R(\nabla_x u_x,Ju_x)Ju_x,u_x\big)u_x\ra dx.
\end{eqnarray}

$I_{31}$ is the Schr\"odinger part:
\begin{eqnarray}
I_{31}&=&2\int\la J\nb_xu_x,R(\nabla_x^3u_x,Ju_x)Ju_x\ra dx
+6\int\la J\nb_xu_x,R(\nabla_x^2u_x,J\nabla_x u_x)Ju_x\ra dx\nn\\
&&{}+2\int\la J\nb_xu_x,R\big(R(\nabla_x u_x,Ju_x)Ju_x,u_x\big)u_x\ra dx.\nn
\end{eqnarray}
Note that the first term of $I_{31}$:
\begin{eqnarray}\label{eq:3.27}
&&{}2\int\la J\nb_xu_x,R(\nabla_x^3u_x,Ju_x)Ju_x\ra dx
=2\int\la \nb_x^3u_x,R(\nabla_xu_x,u_x)Ju_x\ra dx\nn\\
&=&-2\int\la \nb_x^2u_x,R(\nabla^2_xu_x,u_x)Ju_x\ra dx
-2\int\la \nb_x^2u_x,R(\nabla_xu_x,u_x)J\nb_xu_x\ra dx.\nn
\end{eqnarray}
For the second term of $I_{31}$ we have:
\begin{eqnarray}
6\int\la J\nb_xu_x,R(\nabla_x^2u_x,J\nabla_x u_x)Ju_x\ra dx
=6\int\la \nb_x^2xu_x,R(\nabla_xu_x,u_x)J\nabla_x u_x\ra dx.\nn
\end{eqnarray}
Hence we obtain
\begin{eqnarray}\label{I31}
I_{31}&=&-2\int\la \nb_x^2u_x,R(\nabla^2_xu_x,u_x)Ju_x\ra dx
+4\int\la \nb_x^2u_x,R(\nabla_xu_x,u_x)J\nb_xu_x\ra dx\nn\\
&&{}-2\int\la \nb_xu_x,R\big(R(\nabla_x u_x,Ju_x)Ju_x,u_x\big)Ju_x\ra dx.
\end{eqnarray}

In view of (\ref{F1})$-$(\ref{I31}) we have
\begin{eqnarray}\label{eq:3.32}
{dE_4\over dt}&=& \sum_{i=1}^3A_i{dF_i\over dt}=\alpha \sum_{i=1}^3I_{1i}+\beta \sum_{j=1}^3I_{2j}\nn\\
&=&
3(3A_1+2A_2)\beta\int\la \nabla_x^2u_x,R(\nabla_x^2 u_x,u_x)\nabla_x u_x\ra dx\nn\\
&&{}+3(5A_1+2A_3)\beta\int\la \nabla_x^2u_x,R(\nabla_x^2 u_x,Ju_x)J\nabla_x u_x\ra dx\nn\\
&&{}+2(A_1-A_2+A_3)\beta\int\la R(\nabla_x u_x,Ju_x)Ju_x,R(\nabla_x^2 u_x,u_x)u_x\ra dx\nn\\
&&{}-2(A_1-A_2+A_3)\al\int\la \nabla_x^2u_x,R(\nb^2_xu_x,Ju_x)u_x\ra dx\nn\\
&&{}-2(A_1+4A_2+A_3)\al\int\la \nabla_x^2u_x,R(\nb_xu_x,J\nb_xu_x)u_x\ra dx\nn\\
&&{}-4(A_1-A_2)\al\int\la \nabla_x^2u_x,R(\nb_xu_x,Ju_x)\nb_xu_x\ra dx\nn\\
&&{}-2A_2\al\int\langle \nb_xu_x,R(R(\nabla_x u_x,u_x)u_x,u_x)Ju_x\rangle dx\nn\\
&&{}-2A_3\al\int\la \nb_xu_x,R\big(R(\nabla_x u_x,Ju_x)Ju_x,u_x\big)Ju_x\ra dx\nn\\
&&{}+2(A_1-A_2)\beta\int\la R(u_x,Ju_x)Ju_x,R(\nabla_x^2 u_x,\nabla_x u_x)u_x\ra dx\nn\\
&&{}+(A_1+4A_2)\beta\int\la R(u_x,Ju_x)Ju_x,R(\nabla_x^2 u_x,u_x)\nabla_x u_x\ra dx\nn\\
&&{}+(A_1-A_2)\beta\int\la R(u_x,Ju_x)J\nabla_x u_x,R(\nabla_x^2 u_x,u_x)u_x\ra dx\nn\\
&&{}+2A_2\beta\int\langle R(\nabla_x u_x,u_x)u_x,R(\nabla_x^2u_x,u_x)u_x\rangle dx\nn\\
&&{}+A_2\beta\int\langle R(u_x,Ju_x)Ju_x,R(R(\nabla_x u_x,u_x)u_x,u_x)u_x\rangle dx\nn\\
&&{}-A_3\beta\int\la R(u_x,Ju_x)J\nabla_x u_x,R(\nabla_x^2u_x,Ju_x)Ju_x\ra dx\nn\\
&&{}-2A_3\beta\int\la R(\nabla_x u_x,Ju_x)Ju_x,R(\nabla_x^2u_x,Ju_x)Ju_x\ra dx\nn\\
&&{}+2A_3\beta\int\la R(u_x,Ju_x)Ju_x,R(\nabla_x^2u_x,J\nabla_x u_x)Ju_x\ra dx\nn\\
&&{}+A_3\beta\int\la R(u_x,Ju_x)Ju_x,R(\nabla_x u_x,J\nabla_x u_x)J\nabla_x u_x\ra dx\nn\\
&&{}+A_3\beta\int\la R(u_x,Ju_x)Ju_x,R\big(R(\nabla_x u_x,Ju_x)Ju_x,u_x\big)u_x\ra dx.
\end{eqnarray}

Since $A_1=2$, $A_2=-3$ and $A_3=-5$, the first four terms with higher order
derivatives vanish. Let's denote the remaining terms of (\ref{eq:3.32}) by
$G$. It is easy to see that
\begin{eqnarray*}
|G|&\leqslant& C(N)|\al|\int \big(|\nb_x^2u_x||\nb_xu_x|^2|u_x|+|\nb_xu_x|^2|u_x|^4\big)dx \\
&&{}+C(N)|\beta| \int\big(|\nabla_x^2u_x||\nabla_x u_x||u_x|^4+|\nabla_x u_x|^3|u_x|^3+|\nabla_x u_x||u_x|^7\big)dx\\
&\leqslant&  C(N,\al,\beta)\left( ||u_x||_{L^\infty}(\int|\nabla_x^2u_x|^2dx)^{1\over2}
(\int|\nabla_x u_x|^4dx)^{1\over2}+||u_x||^4_{L^\infty}||\nb_xu_x||^2_{L^2}\right.\\
&&{}+||u_x||^4_{L^\infty}(\int|\nabla_x^2u_x|^2dx)^{1\over2}(\int|\nabla_x u_x|^2dx)^{1\over2}\\
&&{}\left.+||u_x||^3_{L^\infty}\int|\nabla_x
u_x|^3dx+||u_x||^6_{L^\infty}(\int|\nabla_x
u_x|^2dx)^{1\over2}(\int|u_x|^2dx)^{1\over2}\right).
\end{eqnarray*}
By the interpolation inequality for sections on vector bundles (see
\cite{YD} for details):
\begin{eqnarray*}
||u_x||_{L^\infty}&\leqslant& C(N)(||\nabla_x u_x||^2_{L^2}+||u_x||^2_{L^2})^{1\over4}||u_x||^{1\over2}_{L^2}\nn\\
&\leqslant& C(N, ||\nabla_x u_x||_{L^2}, E_1(u_0));\label{eq:3.33}\\
||\nabla_x u_x||^3_{L^3}&\leqslant& C(N)(||\nabla_x ^2u_x||^2_{L^2}+||\nabla_x u_x||^2_{L^2})^{1\over4}||\nabla_x u_x||^{5\over2}_{L^2}\nn\\
&\leqslant&C(N, ||\nabla_x u_x||_{L^2})\big(1+||\nabla_x^2u_x||^2_{L^2}\big);\label{eq:3.34}\\
||\nb_xu_x||^4_{L^4}&\leqslant&C(N)(||\nabla_x ^2u_x||^2_{L^2}+||\nabla_x u_x||^2_{L^2})^{1\over2}||\nabla_x u_x||^{3}_{L^2},
\end{eqnarray*}
we have
\begin{eqnarray*}
|G|&\leqslant& C\big(1+||\nabla_x ^2u_x||^2_{L^2}\big),
\end{eqnarray*}
which implies
\begin{eqnarray*}
{dE_4\over dt}&\leqslant& C\big(1+\int|\nabla_x^2u_x|^2dx)\\
&\leqslant& C(1+E_4).
\end{eqnarray*}
where $C=C(N, ||\nabla_x u_x||_{L^2}, E_1(u_0),\al,\beta)$ only depends on $N$, $\al$, $\beta$,
$E_1(u_0)$ and $||\nabla_x u_x||_{L^2}$. This completes the
proof.\hspace*{\fill}$\Box$

\section{Global existence}
In this section we will prove Theorem \ref{thm:1.5}. Since $u_0\in
H^4(S^1, N)$, we can always choose a sequence of smooth maps $u_{0i}
\in C^\infty(S^1, N)$ such that, as $i\rightarrow\infty$,
$$\|u_{0i}-u_0\|_{H^4}\rightarrow 0.$$
From the previous arguments in Theorem 1.3, we know that the Cauchy
problem (\ref{eq:1.15}) with the initial map $u_{0i}$ admits a
unique smooth local solution $u^i$ such that
 $$u^i \in C([0, T(N, \|u_{0i}\|_{H^4})], H^k(S^1, N)))$$
for any $k \geqslant 4$. Obviously, we can see easily that $T(N,
\|u_{0i}\|_{H^4})$ have a uniform lower bound. Hence, letting
$i\rightarrow \infty$, we obtain the local solution to the Cauchy
problem of the Schr\"odinger-Airy flow with the initial map $u_0\in H^4(S^1, N)$.
So, to prove Theorem \ref{thm:1.5}, we only need to consider the
case $u_0$ is a smooth map from $S^1$ into $N$.

Let $u$ be the local smooth solution of (\ref{eq:1.15}) which exists
on the maximal time interval $[0,T)$. We only need to consider the
case where $T<\infty$.

From Lemma \ref{lm:3.1}, we know
that the energy is preserved by the solution $u$, i.e.
$$E_1(u(t))=E_1(u_0),\qquad \text{for any}\quad t\in [0,T).$$
Moreover, by the assumptions on $N$ given in the theorem and
Corollary \ref{lm:3.3} we know that $E_3$ is preserved, that is
$$E_3(u)=\int|\nabla_x u_x|^2dx-{1\over4}\int\langle u_x,R(u_x,Ju_x)Ju_x\rangle dx$$
is a constant $E_3(u_0)$. Thus we have
\begin{eqnarray}\label{eq:5.1}
||\nabla_x u_x||^2_{L^2}&=&E_3(u_0)+{1\over4}\int\langle u_x,R(u_x,Ju_x)Ju_x\rangle dx\nn\\
&\leqslant& E_3(u_0)+C(N)\int|u_x|^4 dx\nn\\
&\leqslant& C(N,E_1(u_0),E_3(u_0)),
\end{eqnarray}
where we used the interpolation inequality
\begin{eqnarray}\label{eq:5.2}
||u_x||^4_{L^4}&\leqslant&(||\nabla_x u_x||^2_{L^2}+||u_x||^2_{L^2})^{1\over2}||u_x||^3_{L^2}\nn\\
&\leqslant&{1\over 2}||\nabla_x u_x||^2_{L^2}+C(E_1(u_0)).\nn
\end{eqnarray}

Thus from Lemma \ref{lm:3.4}, we have that
$${dE_4\over dt}\leqslant C(N,E_1(u_0),E_3(u_0))\big(1+E_4\big).$$
By Gronwall inequality, we get that $E_4(u(t))$ is uniformly
bounded on $[0,T)$. Hence, we obtain
\begin{eqnarray}\label{eq:5.4}
2||\nabla_x^2u_x||^2_{L^2}&=&E_4(u)+3\int\langle\nabla_x u_x,R(\nabla_x u_x,u_x)u_x\rangle dx\nn\\
&&{}+5\int\langle\nabla_x u_x,R(\nabla_x u_x ,Ju_x)Ju_x\rangle dx\nn\\
&\leqslant& C(N,E_4(u_0))+C(N)||u_x||^2_{L^\infty}||\nabla_x u_x||^2_{L^2}.
\end{eqnarray}

In view of (\ref{eq:3.33}), (\ref{eq:5.1}) and the boundedness of $E_4$, we see
that $||\nabla_x^2u_x||_{L^2}$ is uniformly bounded on $[0,T)$. Hence we have
$$\sup_{t\in[0,T)} ||u_x||_{H^2}\leqslant C(N,E_1(u_0),E_3(u_0),E_4(u_0)).$$

It follows from the proof of Theorem \ref{thm:1.1} that for $m>2$
$$\sup_{t\in[0,T)} ||\nabla_x ^m u_x||_{L^2}\leqslant C(N,E_1(u_0), ||\nabla_x u_{0x}||_{L^2},
||\nabla_x^2 u_{0x}||_{L^2},\cdots,||\nabla_x^m u_{0x}||_{L^2}).$$

Thus, if $T$ is finite, we can find a time-local solution $u_1$ of
(\ref{eq:1.15}) and $u_1$ satisfies the initial value condition
$$u_1(x,T-\epsilon)=u(x,T-\epsilon),$$
where $0<\epsilon<T$ is a small number. Then by the local existence
theorem, $u_1$ exists on the time interval
$(T-\epsilon,T-\epsilon+\eta)$ for some constant $\eta>0$. The
uniform bounds on $||u_x||_{H^2}$ and $||\nabla_x^m u_x ||_{L^2}$
(for all $m>2$) implies that $\eta$ is independent of $\epsilon$.
Thus, by choosing $\epsilon$ sufficiently small, we have
$$T_1=T-\epsilon+\eta>T.$$
By the uniqueness result, we have that $u_1(x,t)=u(x,t)$ for all
$t\in[T-\epsilon,T_1)$. Thus we get a solution of the Cauchy problem
(\ref{eq:1.15}) on the time interval $[0,T_1)$,
which contradicts the maximality of $T$.\hspace*{\fill}$\Box$\\

\vspace{2.0cm}
\noindent{Xiaowei Sun}\\
School of Applied Mathematics,\\
Central University of Finance and Economics, \\
Beijing 100081, P.R. China.\\
Email: sunxw@cufe.edu.cn\\\\
Youde Wang\\
Academy of Mathematics and Systems Science\\
Chinese Academy of Sciences,\\
Beijing 100190, P.R. China.\\
Email: wyd@math.ac.cn


\begin{thebibliography}{50}
\setlength{\itemsep}{-3pt}
\small
\bibitem{Ag} G. P. Agrawal, \emph{Nonlinear fiber optics}, Academic Press, 2007.

\bibitem{AF} C. Athorne, A. P. Fordy; \emph{Generalised KdV and MKdV
equations associated with symmetric spaces,} J. Phys. A: Math. Gen.
\textbf{20}(1987), 1377--1386.

\bibitem{BIKT} I. Bejenaru, A. D. Ionescu, C. E. Kenig, D. Tataru;
\emph{Global Schr\"odinger maps in dimensions $d\geq 2$ ($d\geq2$):
small data in the critical Sobolev spaces}, Ann. of Math. (2) 173
(2011), no. 3, 1443--1506.

\bibitem{Uh} N. H. Chang, J. Shatah, K. Uhlenbeck; \emph{Schr\"{o}dinger
maps,} Commun. Pure Appl. Math. \textbf{53}(2000), 590--602.

\bibitem{CT} J. Colliander, M. Keel, G. Staffilani, H. Takaoka, and
T. Tao; \emph{Sharp global well-posedness for KdV and modified KdV
on $\mathbb{R}$ and $\mathbb{T}$,} J. Amer. Math. Soc.
\textbf{16}(2003), 705--749.

\bibitem{CT1}J. Colliander, M. Keel, G. Staffilani, H. Takaoka, and
T. Tao; \emph{Global well-posedness for Schr?dinger equations with
derivative}, SIAM J. Math. Anal. 33 (2001), no. 3, 649--669
(electronic).

\bibitem{CT2}J. Colliander, M. Keel, G. Staffilani, H. Takaoka, and
T. Tao; \emph{A refined global well-posedness result for
Schr\"odinger equations with derivative}, SIAM J. Math. Anal. 34
(2002), no. 1, 64--86 (electronic).

\bibitem{CYL}  X.-J. Chen, J. Yang, and W. K. Lam, \emph{N-soliton solution
for the derivative nonlinear Schr\"odinger equation with
nonvanishing boundary conditions}, J. Phys. A 39 (2006), 3263--3274.

\bibitem{Da} L. S. Da Rios; \emph{On the motion of an unbounded fluid with a
vortex filament of any shape}, Rend. Circ. Mat. Palermo
\textbf{22}(1906), 117--135.

\bibitem{DQW} Q. Ding, Y. D. Wang; \emph{Geometric KdV flows, motions
of curves and the third-order system of the AKNS hierarchy},
Internat. J. Math., \textbf{22}(2011), No.7, 1013--1029.

\bibitem{D} W. Y. Ding; \emph{On the Schr\"odinger flows}, Proc. ICM Beijing 2002, 283--292.

\bibitem{DW}W. Y. Ding, Y. D. Wang; \emph{Schr\"{o}dinger flows of maps
into symplectic manifolds,} Sci. China \textbf{A41}(1998), 746--755.

\bibitem{YD}W. Y. Ding, Y. D. Wang; \emph{Local Schr\"{o}dinger flow into
K\"{a}hler manifolds,} Sci. China \textbf{A44}(2001), 1446--1464.

\bibitem{ES}J. Eells, J. H. Sampson; \emph{Harmonic mappings
of Riemannian manifolds}, Am. J. Math. \textbf{86}(1964), 109--160.

\bibitem{FK} A. P. Fordy, P. P. Kulish; \emph{Nonlinear Schr\"{o}dinger
equations and simple Lie algebras}, Commun. Math. Phys
\textbf{89}(1983), 427--443.

\bibitem{FM} Y. Fukumoto, T. Miyazaki; \emph{Three-dimensional distortions of a
vortex filament with axial velocity,} J. Fluid Mech.
\textbf{222}(1991), 369--416.

\bibitem{GK} V. Gerdjikov, N. A. Kostov; \emph{Reductions of multicomponent mKdV equations
on symmetric spaces of DIII-type}, SIGMA \textbf{4}(2008), 29--58.

\bibitem{Gu} P. Guha; \emph{Geometry of Chen-Lee-Liu type derivative
nonlinear Schr\"odinger flow}, Regular and Chaotic Dynamics,
\textbf{8}(2003), No.2, 213--224.

\bibitem{GKT} S. Gustafson, K. Kang, T. Tsai; \emph{Schr\"odinger flow near
harmonic maps,} Comm. Pure Appl. Math., \textbf{60}(2007), 463--499.

\bibitem{Ha} H. Hasimoto; \emph{A soliton on a vortex filament}, J.  Fluid
mech. \textbf{51}(1972), 477--485.

\bibitem{Hel} S. Helgason; \emph{Differential Geometry, Lie Groups,
and Symmetric Spaces}, AMS, Providence, Rhode Island.

\bibitem{Hi} R. Hirota; \emph{Exact envelope-soliton solutions of a
nonlinear wave equation,} J. Math. Phys. \textbf{14}(1973),
805--809.

\bibitem{HK} A. Hasegawa, Y. Kodama; \emph{Nonlinear pulse propagation
in a monomode dielectric guide}, IEEE J. Quantum Elec.
\textbf{23}(1987), 510--524.

\bibitem{Kod} Y. Kodama; \emph{Optical solitons in a monomode fiber},
J. Statistical Phys. \textbf{39}(1985), 597--614.

\bibitem{KN}  D. J. Kaup and A. C. Newell, \emph{An exact solution for
a derivative nonlinear Schr\"odinger equation}, J. Math. Phys. 19
(1978), 789--801.

\bibitem{L} J. Lenells, \emph{The derivative nonlinear Schr\"odinger
equation on the half-line}, preprint.

\bibitem{Mj} E. Mjolhus, \emph{On the modulational instability of
hydromagnetic waves parallel to the magnetic
field}, J. Plasma Phys. 16 (1976), 321--334.

\bibitem{MR} K. Moffatt, L. Ricca; \emph{Interpretation of invariants of the
Betchov-Da rios equation and of the euler equations}, The Global Geometry of
Turbulence, Plenum Press, New York (1991).

\bibitem{NT}T. Nishiyama, A. Tani; \emph{Initial and initial-boundary value
problems for a vortex filament with or without axial flow,} SIAM J.
Math. Anal. \textbf{27}(1996), 1015--1023.

\bibitem{TN} A. Tani, T. Nishiyama; \emph{Solvability of equations for
motion of a vortex filament with or without axial flow}, Publ. Res.
Inst. Math. Sci. \textbf{33}(1997), 509--526.

\bibitem{Od} E. Onodera; \emph{A third-order dispersive flow for closed curves
into K\"ahler manifolds}, J. Geom. Anal. 18 (2008), no. 3, 889--918.

\bibitem{Od1} E. Onodera; \emph{A Remark on the global existence of a third o
rder dispersive flow into locally Hermitian symmetric spaces}, Comm.
Partial Differential Equations 35 (2010), no. 6, 1130--1144.

\bibitem{PW} P. Y. Y. Pang, H. Y. Wang, Y. D. Wang; \emph{Schr\"{o}dinger
flow on Hermitian locally symmetric spaces}, Comm. Anal. Geom.
\textbf{10}(2002), 653--681.

\bibitem{Pe} P. Peterson; \emph{Riemannian Geometry}, New York:
Springer Verlag(1998).

\bibitem{RL}R. Ricca; \emph{Rediscovery of Da Rios equations,} Nature \textbf{352},
561--562(1991).

\bibitem{RRS} I. Rodnianski, Y. A. Rubinstein, G. Staffilani;
\emph{On the global well-posedness of the one-dimensional Schrodinger map flow,}
Analysis and PDE, \text{2}(2009), 187--209.

\bibitem{Sa} D. Sattinger; \emph{Hamiltonian hierarchies on semi-simple Lie
algebras}, Stud. Appl. Math., \textbf{72} (1984), 65--86.

\bibitem{SW} X. W. Sun, Y. D. Wang; \emph{KdV Geometric Flows on K\"ahler
Manifolds}, Internat. J. Math., \textbf{22}(2011), No.10, 1439-1500.

\bibitem{Tak} H. Takaoka; \emph{Well-posedness for the one-dimensional nonlinear
Schr\"odinger equation with the derivative nonlinearity}, Adv. Diff.
Equations 4 (1999), no. 4, 561--580.

\bibitem{TA}M. E. Taylor; \emph{Partial Differential Equations III: Nonlinear Equations},
New York: Springer Verlag(1997).

\bibitem{TU}C. Terng, K. Uhlenbeck; \emph{Schr\"odinger flows on Grassmannians,}
Integrable systems, geometry, and topology, 235--256, AMS/IP Stud.
Adv. Math., \textbf{36}, Amer. Math. Soc., Providence, RI, 2006.

\bibitem{WH}B. Wang, L. Han, C. Huang; \emph{Global well-posedness and
scattering for the derivative nonlinear Schr\"odinger equation with
small rough data}, Ann. Inst. H. Poincaré Anal. Non Lin\'eaire
26(2009), no.6, 2253--2281.

\bibitem{W}Y. D. Wang; \emph{Lecture on geometric flows on K\"ahler manifolds}(Part I), manuscript.
\end{thebibliography}
\end{document}